\documentclass[10pt]{amsart}
\usepackage{amssymb,amsmath,amsthm}
\usepackage{graphicx,xcolor}
\usepackage[foot]{amsaddr}
\usepackage[textwidth=18cm,textheight=24cm,centering]{geometry}
\usepackage{tikz,tikz-cd}
\usetikzlibrary{arrows.meta}
\usepackage{caption}
\usepackage{hyperref}
\usepackage[style=alphabetic,backend=biber]{biblatex}
\newcommand{\sh}[2]{\langle #1,#2\rangle}
\allowdisplaybreaks
\newcommand{\gr}{\mathrm{gr}}
\newcommand{\af}{\mathrm{af}}

\theoremstyle{plain}
\newtheorem{theorem}{Theorem}
\newtheorem{lemma}[theorem]{Lemma}
\newtheorem{proposition}[theorem]{Proposition}
\newtheorem{corollary}[theorem]{Corollary}

\theoremstyle{definition}
\newtheorem{definition}[theorem]{Definition}
\newtheorem{assumption}[theorem]{Assumption}
\newtheorem{example}[theorem]{Example}
\theoremstyle{remark}
\newtheorem{remark}[theorem]{Remark}
\theoremstyle{plain}
\newtheorem{theoremA}{Theorem}

\newtheorem{theoremB}{Theorem}

\newtheorem{theoremC}{Theorem}

\newtheorem{theoremD}{Theorem}

\newcommand{\fg}{{\mathfrak g}}
\newcommand{\fsl}{{\mathfrak{sl}}}
\newcommand{\fb}{{\mathfrak b}}
\newcommand{\fp}{{\mathfrak p}}
\newcommand{\fh}{{\mathfrak h}}
\newcommand{\fn}{{\mathfrak n}}
\newcommand{\fa}{{\mathfrak a}}
\newcommand{\fc}{{\mathfrak c}}
\newcommand{\Iw}{{\mathcal I}}
\newcommand{\Ofd}{{\mathcal{O}^{\mathrm{fd}}}}
\newcommand{\Dalpha}{D(2|1,\alpha)}
\newcommand{\Ddeg}{D(2|1,\mathrm{deg})}
\newcommand{\Daff}{D^{(1)}(2|1,\mathrm{deg})}
\newcommand{\Span}[1]{\operatorname{span}\left\{#1\right\}}

\newcommand{\Lind}{\mathbf{L}\widetilde{\Ind}}
\newcommand{\Rcoind}{\mathbf{R}\widetilde{\coInd}}
\DeclareMathOperator{\Res}{Res}
\DeclareMathOperator{\Ind}{Ind}
\DeclareMathOperator{\ch}{ch}
\DeclareMathOperator{\coInd}{coInd}
\DeclareMathOperator{\Hom}{Hom}
\DeclareMathOperator{\U}{U}
\DeclareMathOperator{\Ext}{Ext}
\newcommand{\Zuck}{\mathcal{Z}}
\newcommand{\Kom}{\operatorname{Kom}}

\newcommand{\adj}[2]{#1\oplus\Span{#2}}
\newcommand{\ee}[1]{e_{#1}}
\newcommand{\J}[1]{\ensuremath{[#1]}}

\newcommand{\rd}[1]{{\color{red}#1}}

\newcommand{\vdefeq}[1]{\substack{\rotatebox[origin=c]{-90}{$\scriptstyle=:$}\\[2pt] #1}}

\definecolor{evenroot}{RGB}{236, 72,153} % pink   : even roots
\definecolor{oddroot}{RGB}{250,204, 21}  % yellow : odd simple roots
\definecolor{oddtext}{RGB}{133, 77, 14}  % dark gold : odd labels
\definecolor{eventext}{RGB}{190, 24,93}  % dark pink : even labels
\definecolor{oddblue}{RGB}{ 59,130,246}  % blue : non-simple odd roots
\definecolor{oddbluetext}{RGB}{30, 64,175} % dark blue : their labels
\tikzset{
	rootfig/.style={
		x={(1.2cm,0cm)}, y={(0.54cm,0.66cm)}, z={(0cm,1.2cm)},
		line cap=round, line join=round,
		lab/.style={font=\small},
		wlab/.style={lab, anchor=west,        xshift= 5pt},
		elab/.style={lab, anchor=east,        xshift=-5pt},
		nwlab/.style={lab, anchor=north west, xshift= 4pt, yshift=-3pt},
		nelab/.style={lab, anchor=north east, xshift=-4pt, yshift=-3pt},
		selab/.style={lab, anchor=south east, xshift=-4pt, yshift= 3pt},
		slab/.style={lab, anchor=south,       yshift= 3pt},
		nlab/.style={lab, anchor=north,       yshift=-3pt},
		axis/.style={evenroot!70, line width=0.7pt},
		axistip/.style={axis, -{Latex[length=2.8mm,width=2.2mm]}},
		axishidden/.style={axis, dash pattern=on 2pt off 2.4pt},
		hidden/.style={black!45, dash pattern=on 2.5pt off 2pt},
		solid edge/.style={black!75, line width=0.6pt},
	}
}

\newcommand{\affil}[1]{\textsuperscript{#1}}
\makeatletter
\newcommand{\affiltext}[2]{
	\begingroup
	\renewcommand{\thefootnote}{#1}
	\phantomsection
	\footnotetext{#2}
	\endgroup
}
\makeatother
\title{Categorification of the genus two DAHA}
\author{Semeon~Arthamonov\affil{1}}
\author{Ievgen~Makedonskyi\affil{1}}
\author{Daniil~Sarafannikov\affil{2}}

\begin{document}
	
	\begin{abstract}
		We construct three derived endofunctors of the bounded derived category of finite-dimensional graded representations of the current algebra of a flat degeneration of $\Dalpha$, and prove that their classes in the Grothendieck group are, up to the factor $t$, the genus two Macdonald operators. Computing the graded $\Ext$ pairing explicitly, we deduce the self-adjointness of the genus two Macdonald operators and the orthogonality of the genus two Macdonald polynomials with respect to a certain sesquilinear form.
	\end{abstract}
	
	\maketitle
	
	\affiltext{1}{Beijing Institute of Mathematical Sciences and Applications (BIMSA), Beijing, China. \textit{Email:} arthamonov@bimsa.cn, mak@bimsa.cn}
	\affiltext{2}{Qiuzhen College, Tsinghua University, Beijing, China. \textit{Email:} sardanis123@gmail.com}
	
	\section{Introduction}
	
	The double affine Hecke algebra of Cherednik~\cite{Cherednik-2005} and its spherical subalgebra provide the algebraic framework for the theory of Macdonald polynomials. In the rank-one case the spherical DAHA of type $A_1$ acts on the ring of symmetric Laurent polynomials in one variable, the Macdonald difference operator is one of its generators and the Macdonald polynomials are the eigenfunctions of this operator. From the topological point of view, this algebra is attached to the once-punctured torus: its mapping class group $SL(2,\mathbb{Z})$ acts on the algebra by automorphisms.
	
	In~\cite{ArthamonovShakirov-2019} a genus two counterpart of this picture was proposed. The construction of~\cite{ArthamonovShakirov-2019} starts from three commuting $q$-difference operators $\hat O_{A_{12}}$, $ \hat O_{A_{13}}$, $\hat O_{A_{23}}$ in three variables $X_1$, $X_2$, $X_3$ depending on two parameters $q$ and $t$. The algebra generated by these operators together with the operators of multiplication by $X_k+X_k^{-1}$ carries an action of the Dehn twists of the closed genus two surface by automorphisms, and these automorphisms satisfy all the relations of the mapping class group of that surface. The joint eigenfunctions of the three commuting operators are the genus two analogues of the $A_1$ Macdonald polynomials. Following~\cite{ArthamonovShakirov-2019} we call $\hat O_{A_{ij}}$ the \emph{genus two Macdonald operators}. The label $A_{ij}$ refers to a curve on the genus two surface.
	
	The purpose of the present paper is to categorify these operators, that is, to realize them as the classes of a family of derived functors between categories of representations. The mechanism we use is similar to the one which was used in~\cite{KKhM}: one takes a Lie superalgebra with a suitable family of parabolic subalgebras, considers the induction, restriction and coinduction functors between the corresponding categories of graded representations, and computes their action on the Grothendieck groups, which are identified with rings of Laurent polynomials by means of characters. A similar categorification was done in~\cite{bezrukavnikov2012affine} using a geometric approach. It would be interesting to obtain a geometric analogue of our results.
	
	The algebra in question is the loop superalgebra of a degeneration of the exceptional family of simple Lie superalgebras $\Dalpha$. The superalgebra $\Dalpha$ has dimension $(9|8)$, its even part is $\fsl_2\oplus\fsl_2\oplus\fsl_2$ and its rank is three. The whole dependence on the parameter $\alpha$ is concentrated in the component $S^2\fg_{\overline 1}\rightarrow \fg_{\overline 0}$ of the bracket, so that setting this component to zero produces a Lie superalgebra $\Ddeg$ which does not depend on $\alpha$. It is isomorphic to $D(2|1,0,0,0)$, see Remark~\ref{rem:DdegIso}. We denote by $\Daff$ the loop superalgebra of $\Ddeg$. It has four simple roots $\alpha_0,\alpha_1,\alpha_2,\alpha_3$, all of them odd.
	
	The degeneration is what makes the algebra convenient to work with, and it is also what creates the additional grading we need. Namely, $\Daff$ carries three compatible gradings: the weight grading with respect to the Cartan subalgebra, with lattice $\mathbb{Z}^3$, the grading by the degree in the loop variable $z$, and the $\mathbb{Z}$-grading in which the even part sits in degree $0$ and the odd part in degree $1$, which exists precisely because the bracket of two odd elements vanishes. We work with finite-dimensional representations graded with respect to all three gradings. They form the categories $\Ofd(\fp_J)$ of Definition~\ref{def:CategoryO}, and $M\sh{l}{r}$ denotes $M$ with the $z$-grading shifted by $l$ and the odd grading by $r$. We denote the corresponding formal variables by $X_1,X_2,X_3$, by $q$ and by $-t$. The variables $X_k$ are naturally labelled by pairs: $X_k=X^{(\alpha_i+\alpha_j)/2}$ for $\{i,j,k\}=\{1,2,3\}$, which is exactly the labelling of the variables of~\cite{ArthamonovShakirov-2019}. Thus the parameters of~\cite{ArthamonovShakirov-2019} acquire a representation-theoretic meaning: their $q$ is our $q$, and their $t$ is our $t^2$.
	
	The parabolic subalgebras of a Kac--Moody superalgebra are defined by choosing a subset of the simple roots and adjoining the corresponding negative root vectors to the Borel subalgebra. In the degenerate case, this recipe is not available, and we give a direct construction of the degenerate parabolic subalgebras $\fp_J$, $J\subset\{0,1,2,3\}$, in Definition~\ref{def:parabolicSubalgebras}. The main object of the paper is the parabolic subalgebra
	\[\fp_{123}\simeq \Ddeg\otimes\Bbbk[z]\]
	together with its three maximal parabolic subalgebras $\fp_{12},\fp_{13},\fp_{23}$. The restriction functor $\Res_{\fp_{ij}}^{\fp_{123}}$ has a left and a right adjoint, the integrable induction and coinduction functors:
	\begin{equation*}
		\begin{tikzcd}[column sep=4.2cm]
			\Ofd(\fp_{123})
			\arrow[r, "{\Res_{\fp_{ij}}^{\fp_{123}}}" description]
			& \Ofd(\fp_{ij})
			\arrow[l, bend right=17, "{\widetilde{\Ind}_{\fp_{ij}}^{\fp_{123}}}"']
			\arrow[l, bend left=17, "{\widetilde{\coInd}_{\fp_{ij}}^{\fp_{123}}}"]
		\end{tikzcd}
	\end{equation*}
	
	The key property of the pair $\fp_{ij}\subset\fp_{123}$ is that the derived induction and coinduction differ only by a shift of the gradings and of the homological degree.
	
	\begin{theoremA}
		There is an isomorphism of functors $D^b(\Ofd(\fp_{ij}))\to D^b(\Ofd(\fp_{123}))$
		\begin{equation}\label{eq:IndCoind}
			{\Rcoind}_{\fp_{ij}}^{\fp_{123}}\simeq {\Lind}_{\fp_{ij}}^{\fp_{123}}\sh{0}{-2}[-2].
		\end{equation}
	\end{theoremA}
	
	The same property holds for the pairs $\fb\otimes\Bbbk[\xi]\subset\fp_i\otimes\Bbbk[\xi]$, where $\fb$ and $\fp_i$ are a Borel and a minimal parabolic subalgebra of a semisimple Lie algebra and $\xi$ is an odd variable. In both cases this property comes from the fact that the sums of the even roots and of the odd roots of the larger algebra that do not lie in the smaller one coincide. We do not know how large the class of pairs with this property is.
	
	The four simple roots of $D^{(1)}(2|1,\alpha)$ are permuted by the symmetric group $\mathfrak{S}_4$. For the non-degenerate algebra these maps change the parameter $\alpha$ and hence are not automorphisms. After the degeneration the dependence on $\alpha$ disappears, and we obtain an action of $\mathfrak{S}_4$ on $\Daff$ by automorphisms. Twisting by an automorphism is an exact functor, and the automorphism $\pi_{(i,j)(k,0)}$, where $\{i,j,k\}=\{1,2,3\}$, preserves the parabolic subalgebra $\fp_{ij}$, so twisting by it, with the weight and the $z$-degree transported by $\pi_{(i,j)(k,0)}$, gives an exact endofunctor of $\Ofd(\fp_{ij})$. The functors we study are the compositions
	\begin{gather*}
		\mathcal{F}_{ij}:=\Lind_{\fp_{ij}}^{\fp_{123}} \circ \pi_{(i,j)(k,0)}\circ {\Res}_{\fp_{ij}}^{\fp_{123}}\colon\\
		D^b(\Ofd(\fp_{123}))\xrightarrow{\Res}D^b(\Ofd(\fp_{ij}))\xrightarrow{\pi_{(i,j)(k,0)}}D^b(\Ofd(\fp_{ij}))\xrightarrow{\Lind}D^b(\Ofd(\fp_{123})),
	\end{gather*}
	see Definition~\ref{def:Fdefinition}. Our main categorification result is the following.
	
	\begin{theoremB}
		\[[\mathcal{F}_{ij}]=t\, \hat O_{A_{ij}}.\]
	\end{theoremB}
	
	The operators $\hat O_{A_{ij}}$ commute~\cite{ArthamonovShakirov-2019}, hence so do the classes $[\mathcal{F}_{ij}]$. We lift this to an isomorphism of functors.
	
	\begin{theoremC}\label{thm:CommFunctors}
		For all pairwise distinct $i,j,k\in\{1,2,3\}$ there is an isomorphism of endofunctors of $D^b(\Ofd(\fp_{123}))$
		\[\mathcal{F}_{ij}\circ\mathcal{F}_{ik}\simeq \mathcal{F}_{ik}\circ \mathcal{F}_{ij}.\]
	\end{theoremC}
	
	The category $\Ofd(\fp_{123})$ carries a natural graded $\Ext$ pairing
	\[(M,N)_{\Ext}:=\sum_{i=0}^{\infty}(-1)^i\dim_{q,t}\bigl(\Ext^{i}_{\gr}(M ,N)\bigr),\]
	see Subsection~\ref{ssec:ExtPairing}. This form is additive with respect to short exact sequences in each argument, hence descends to a pairing on the Grothendieck group, and it is skew-linear in the second argument with respect to the substitution $f^\star=f|_{t\mapsto t^{-1},q\mapsto q^{-1}}$. Under the identification of $K_0(\Ofd(\fp_{123}))$ with the ring $\mathbb{Z}[X_1^{\pm 1},X_2^{\pm 1},X_3^{\pm 1}]^{s_1,s_2,s_3}[t^{\pm 1},q^{\pm\frac{1}{2}}]$, where $s_i$ acts by $X_i\mapsto X_i^{-1}$, by means of characters, the pairing takes the explicit form
	\[(f,g):=\bigl(f\,\Delta\, g^\star\bigr)_0,\qquad \Delta:=\frac18\,(q;q)_\infty^3\frac{\prod_{l \geq 0}\prod_{i=1}^3(1-X_i^2 q^l)(1-X_i^{-2} q^l)}{\prod_{l \geq 0}\prod_{a,b,c\in\{\pm 1\}}(1-X_1^aX_2^bX_3^c q^lt)},\]
	where $(\cdot)_0$ denotes the constant term with respect to the variables $X_i$, see Lemma~\ref{lem:ExtPairingExplicit}. The density $\Delta$ is the genus two counterpart of the Macdonald weight function: the numerator is the contribution of the even part of the algebra, and the denominator is the contribution of the odd part.
	
	Together with the adjunctions and $\pi_{(i,j)(k,0)}^2=\mathrm{id}$, Theorem~A implies that a suitably normalized version of $\mathcal{F}_{ij}$ is self-adjoint with respect to this pairing, namely $\mathcal{F}_{ij}\sh{0}{-1}[-1]$:
	\[\Hom_{\gr}(\mathcal{F}_{ij}M\sh{0}{-1}[-1],N)\simeq \Hom_{\gr}(M,\mathcal{F}_{ij}N\sh{0}{-1}[-1]),\qquad [\mathcal{F}_{ij}\sh{0}{-1}[-1]]=\hat O_{A_{ij}},\]
	and passing to the Grothendieck group we obtain the following property of the genus two Macdonald operators.
	
	\begin{theoremD}
		For all $f,g\in K_0(\Ofd(\fp_{123}))$ one has
		\[\bigl(\hat O_{A_{ij}}f\cdot \Delta\, g^{\star}\bigr)_0
		=\bigl(f\cdot \Delta\,(\hat O_{A_{ij}}g)^{\star}\bigr)_0.\]
	\end{theoremD}

	As a consequence, the joint eigenfunctions $\Psi_{i_1,i_2,i_3}$ of $\hat O_{A_{12}},\hat O_{A_{13}},\hat O_{A_{23}}$ are pairwise orthogonal with respect to $\Delta$. They were constructed in~\cite{ArthamonovShakirov-2019}, are indexed by the triples $(i_1,i_2,i_3)$ of non-negative integers admissible in the sense of~\cite{ArthamonovShakirov-2019} and specialize to the $A_1$ Macdonald polynomials. For $(i_1,i_2,i_3)\neq(j_1,j_2,j_3)$ one has
	\[\bigl( \Psi_{i_1,i_2,i_3}\,\Delta\, \Psi^{\star}_{j_1,j_2,j_3} \bigr)_0=0,\]
	see Corollary~\ref{cor:Orthogonality}. This is the genus two analogue of the orthogonality of Macdonald polynomials with respect to the Macdonald weight.
	
	The paper is organized as follows. In Section~\ref{sec:LieSuperalgebras} we recall generalities on Lie superalgebras, their root systems and parabolic subalgebras, and then introduce $\Ddeg$, its loop superalgebra $\Daff$, the parabolic subalgebras $\fp_J$ and the diagrammatic automorphisms. Section~\ref{sec:Categories} contains the homological preliminaries, the definition of the categories of representations we work with and the computation of their Grothendieck groups in terms of characters. Section~\ref{sec:Functors} deals with the induction, restriction and coinduction functors, first in general and then for the parabolic subalgebras of $\Daff$. The main computation there is Lemma~\ref{lem:InductionOperator}. Section~\ref{sec:Categorification} recalls the genus two Macdonald operators of~\cite{ArthamonovShakirov-2019} and contains the categorification theorem. Section~\ref{sec:Orthogonality} introduces the graded $\Ext$ pairing, computes it explicitly, and contains the proof of Theorem~A and the orthogonality statements. Section~\ref{sec:ProofC} contains the proof of Theorem~\ref{thm:CommFunctors}.
	
	\section{Lie superalgebras}\label{sec:LieSuperalgebras}
	
	Throughout the paper $\Bbbk$ denotes an algebraically closed field of characteristic zero. Unless stated otherwise, all vector spaces, algebras and tensor products are taken over $\Bbbk$.
	
	In this section we first recall the general notions concerning Lie superalgebras, their root systems and parabolic subalgebras. We then specialize to the Lie superalgebra which is the main object of this paper, namely to a flat degeneration of $\Dalpha$ and to its loop superalgebra. We do not prove any of the general statements about Lie superalgebras here. The standard references are~\cite{Mus} and~\cite{DictionaryOn}.
	
	\subsection{Generalities}
	Recall that a Lie superalgebra is a $\mathbb{Z}/2\mathbb{Z}$-graded $\Bbbk$-linear space $\fg = \fg_{\overline{0}}\oplus\fg_{\overline{1}}$ together with a bilinear map $\left[\cdot ,\cdot \right]:\fg\otimes\fg\rightarrow\fg$ such that $\left[\fg_{\overline{i}}, \fg_{\overline{j}}\right]\subset\fg_{\overline{i+j}}$ and such that for all homogeneous elements $X, Y, Z \in \fg$ the following identities hold:
	\begin{gather*}
		\left[X,Y\right]=-(-1)^{\deg{X}\deg{Y}}\left[Y,X\right],\\
		\left[X,\left[Y,Z\right]\right]=\left[\left[X,Y\right],Z\right]+(-1)^{\deg{X}\deg{Y}}\left[Y,\left[X,Z\right]\right].
	\end{gather*}
	Here and below $\deg X \in \mathbb{Z}/2\mathbb{Z}$ denotes the parity of a homogeneous element $X$, and all the identities involving parities are stated for homogeneous elements and extended by linearity.
	
	In particular, $\fg_{\overline 0}$ is an ordinary Lie algebra and $\fg_{\overline 1}$ is a $\fg_{\overline 0}$-module. A module over a Lie superalgebra is always assumed to be $\mathbb{Z}/2\mathbb{Z}$-graded, and the action is assumed to respect this grading. We denote by $\U(\fg)$ the universal enveloping algebra of $\fg$. It is an associative superalgebra, and the Poincar\'e--Birkhoff--Witt theorem holds for it in the form
	\[\U(\fg)\simeq \U(\fg_{\overline 0})\otimes \Lambda(\fg_{\overline 1})\]
	as $\mathbb{Z}/2\mathbb{Z}$-graded vector spaces, see~\cite[Chapter~6]{Mus}.
	
	The Lie superalgebras whose modules we consider, the parabolic subalgebras $\fp_J$ with $J\neq\{0,1,2,3\}$ and the subalgebras between them, admit a semidirect decomposition
	\begin{equation}\label{eq:ReductivePart}
		\fg \simeq \fg^{\mathrm{red}}\ltimes R,
	\end{equation}
	where $R$ is a residually nilpotent ideal, that is, the intersection of its lower central series is zero, and $\fg^{\mathrm{red}}$ is a reductive Lie algebra contained in $\fg_{\overline 0}$. We call $\fg^{\mathrm{red}}$ the \emph{reductive part} of $\fg$ and $R$ the \emph{radical} of $\fg$. This decomposition is part of the data and is fixed once and for all for each algebra we consider.
	
	\subsection{Root systems, Borel and parabolic subalgebras}
	Let $\fh \subset \fg_{\overline 0}$ be a Cartan subalgebra of the reductive part $\fg^{\mathrm{red}}$ and assume that $\fh$ acts semisimply on $\fg$ by the adjoint action. Then we have the root decomposition
	\[\fg=\fh \oplus \bigoplus_{\gamma \in \Phi}\fg_{\gamma},\qquad \fg_\gamma:=\{X \in \fg \mid [H,X]=\gamma(H)X ~\text{for all}~ H \in \fh\},\]
	and the root system $\Phi \subset \fh^*$ splits into the even and the odd parts
	\[\Phi=\Phi_{\overline 0}\sqcup \Phi_{\overline 1},\qquad \Phi_{\overline i}:=\{\gamma \in \Phi\mid \fg_\gamma \cap \fg_{\overline i}\neq 0\}.\]
	A choice of a Borel subalgebra $\fb \supset \fh$ is the same thing as a choice of a subset $\Phi_+\subset \Phi$ of positive roots.
	
	In contrast to the case of semisimple Lie algebras, for a Lie superalgebra, different systems of simple roots need not be conjugate to each other, so that a Lie superalgebra has in general several non-equivalent Dynkin diagrams, see~\cite[Chapter~3]{DictionaryOn} and~\cite{DrinfieldSecond}. In this paper we do not need this flexibility: we fix one particular system of simple roots, described in Subsection~\ref{ssec:Ddeg} below, and all the constructions of the paper refer to this fixed choice.
	
	\begin{definition}\label{def:WeylGroupConvention}
		Throughout the paper the \emph{Weyl group} $W$ of a Lie superalgebra $\fg$ means the Weyl group of the reductive part $\fg^{\mathrm{red}}$, that is, the group generated by the reflections $s_\gamma$ for the roots $\gamma$ of $\fg^{\mathrm{red}}$.
	\end{definition}
	
	Note that $W$ acts on $\fh^*$ and preserves $\Phi$, but it does not act transitively on the set of systems of simple roots. Dominance of a weight refers to this group.
	
	Finally, we use the following notational convention for parabolic subalgebras: for a subset $J$ of the set of indices of the simple roots, we write $\fp_J$, always omitting the braces, so that $\fp_{ij}:=\fp_{\{i,j\}}$ and $\fp_{123}:=\fp_{\{1,2,3\}}$. The precise definition of $\fp_J$ in the situation we need is given in Definition~\ref{def:parabolicSubalgebras}. As explained in Remark~\ref{rem:ParabolicDegeneration}, the naive definition used in the semisimple case is not appropriate here.
	
	\subsection{The Lie superalgebra \texorpdfstring{$\Ddeg$}{Ddeg}}\label{ssec:Ddeg}
	Consider the Lie superalgebra $\fg=\fg_{\overline 0}\oplus \fg_{\overline 1}=\Dalpha$ of dimension $(9|8)$. Its even part is isomorphic to $\fsl_2 \oplus \fsl_2 \oplus \fsl_2$. Its weight lattice is isomorphic to $\mathbb{Z}^3$. We denote by $\varepsilon_1, \varepsilon_2,\varepsilon_3$ the natural basis of this lattice. The positive even roots are $2 \varepsilon_i$, $i=1,2,3$.
	
	The odd part $\fg_{\overline 1}$ is isomorphic to the tensor product of the three fundamental representations, $V_{\varepsilon_1}\otimes V_{\varepsilon_2}\otimes V_{\varepsilon_3}$.
	
	For generic values of the parameter $\alpha$, the Lie superalgebra $\Dalpha$ is simple. In this paper we deal with a flat degeneration of this algebra, which we denote by $\Ddeg$. To define it, introduce a parameter $\tau$ and consider the flat family of multiplications on the underlying superspace
	\begin{equation*}
		[X,Y]_\tau:=\begin{cases}
			[X,Y], &\text{ if }X \in  \Dalpha_{\overline 0}\text{ or }Y \in  \Dalpha_{\overline 0},\\
			\tau [X,Y], &\text{ if }X,Y \in  \Dalpha_{\overline 1}.
		\end{cases}
	\end{equation*}
	One checks directly that $[\cdot,\cdot]_\tau$ is a Lie superbracket, i.e.\ that it satisfies the Jacobi identity and super-antisymmetry.
	For $\tau \neq 0$ this bracket is isomorphic to the original one, so it defines a simple Lie superalgebra. The case of interest for us is the special point $\tau=0$.
	\begin{definition}\label{def:degeneratedLieSuperalgebra}
		We write $\Ddeg$ for the Lie superalgebra with the same underlying graded vector space as $\Dalpha$ and with the bracket $[\cdot,\cdot]_0$.
	\end{definition}
	\begin{remark}\label{rem:DdegIso}
		Let $D(2|1,a_1,a_2,a_3)$, $a_1+a_2+a_3=0$, be the family of Lie superalgebras with the same even part and the same action of it on the odd part as $\Dalpha$, in which the bracket of two odd elements is $\sum_{i=1}^3a_ip_i$, where $p_i\colon S^2\fg_{\overline 1}\to\fsl_2^{(i)}$ is a fixed non-zero invariant map to the $i$-th copy of $\fsl_2$. The algebra $\Dalpha$ belongs to this family. At $a_1=a_2=a_3=0$ the bracket of two odd elements vanishes, so the degenerate Lie superalgebra $\Ddeg$ is isomorphic to $D(2|1,0,0,0)$.
	\end{remark}
	
	This Lie superalgebra carries a natural $\mathbb{Z}$-grading in which $\Ddeg_{\overline 0}$ sits in degree $0$ and $\Ddeg_{\overline 1}$ sits in degree $1$. Note that the weights of the basis elements of $\fg_{\overline 1}$
	are $\pm \varepsilon_1 \pm \varepsilon_2 \pm \varepsilon_3$. Introduce the notation
	\begin{gather*}
		\alpha_1=-\varepsilon_1+\varepsilon_2 +\varepsilon_3,\\
		\alpha_2=\varepsilon_1-\varepsilon_2 +\varepsilon_3,\\
		\alpha_3=\varepsilon_1+\varepsilon_2 -\varepsilon_3.
	\end{gather*}
	
	We choose these roots to be the simple positive roots. In particular, all the simple roots of $\Ddeg$ are odd. This is the system of simple roots which is fixed once and for all throughout the paper. We also write $\theta:=\varepsilon_1+\varepsilon_2+\varepsilon_3=\alpha_1+\alpha_2+\alpha_3$ for the highest root. The whole root system is shown in Figure~\ref{fig:RootSystem}.
	\begin{figure}[ht]
		\centering
		\begin{tikzpicture}[rootfig]
			% ---------------- title -------------------------------------------
			% ------- axes: solid outside the cube, dashed inside -------------
			% the arrow tips are put on the outer ends only, so each half axis
			% is drawn from the face of the cube outwards
			\draw[axistip] (-1,0,0) -- (-2.6,0,0);
			\draw[axistip] ( 1,0,0) -- ( 2.6,0,0);
			\draw[axistip] (0,-1,0) -- (0,-3.5,0);
			\draw[axistip] (0, 1,0) -- (0, 3.5,0);
			\draw[axistip] (0,0,-1) -- (0,0,-2.6);
			\draw[axistip] (0,0, 1) -- (0,0, 2.6);
			% the parts running through the interior of the cube
			\draw[axishidden] (-1,0,0) -- (1,0,0)
			(0,-1,0) -- (0,1,0)
			(0,0,-1) -- (0,0,1);
			
			\fill[black!70] (0,0,0) circle (0.8pt);
			
			% ------- the cube ------------------------------------------------
			\draw[hidden]
			(-1,1,-1) -- ( 1,1,-1)
			(-1,1,-1) -- (-1,-1,-1)
			(-1,1,-1) -- (-1, 1, 1);
			\draw[solid edge]
			(-1,-1,-1) -- (1,-1,-1) -- (1,-1,1) -- (-1,-1,1) -- cycle
			(-1,-1, 1) -- (-1,1,1) -- (1,1,1) -- (1,-1,1)
			( 1,-1,-1) -- ( 1,1,-1) -- (1,1,1);
			
			% ------- even roots ----------------------------------------------
			\foreach \p in {(2,0,0),(-2,0,0),(0,2,0),(0,-2,0),(0,0,2),(0,0,-2)}
			\filldraw[fill=evenroot, draw=black!60, line width=0.5pt] \p circle (2.2pt);
			
			\node[eventext, slab]  at ( 2,0,0) {$2\varepsilon_1$};
			\node[eventext, nlab]  at (-2,0,0) {$-2\varepsilon_1$};
			\node[eventext, selab, xshift=2pt, yshift=2pt]   at (0, 2,0) {$2\varepsilon_2$};
			\node[eventext, nwlab, xshift=-8pt, yshift=-2pt] at (0,-2,0) {$-2\varepsilon_2$};
			\node[eventext, wlab]  at (0,0, 2) {$2\varepsilon_3$};
			\node[eventext, wlab]  at (0,0,-2) {$-2\varepsilon_3$};
			
			% ------- odd roots -----------------------------------------------
			\coordinate (ppp) at ( 1, 1, 1);
			\coordinate (ppm) at ( 1, 1,-1);
			\coordinate (pmp) at ( 1,-1, 1);
			\coordinate (pmm) at ( 1,-1,-1);
			\coordinate (mpp) at (-1, 1, 1);
			\coordinate (mpm) at (-1, 1,-1);
			\coordinate (mmp) at (-1,-1, 1);
			\coordinate (mmm) at (-1,-1,-1);
			
			\foreach \p in {mpp,pmp,ppm}
			\filldraw[fill=oddroot, draw=black!60, line width=0.4pt] (\p) circle (2.2pt);
			\foreach \p in {ppp,pmm,mpm,mmp,mmm}
			\filldraw[fill=oddblue, draw=black!60, line width=0.4pt] (\p) circle (2.2pt);
			
			\node[oddbluetext, wlab]  at (ppp) {$\theta$};
			\node[oddtext, wlab]      at (ppm) {$\alpha_3$};
			\node[oddtext, nwlab]     at (pmp) {$\alpha_2$};
			\node[oddbluetext, wlab]  at (pmm) {$-\alpha_1$};
			\node[oddtext, slab]      at (mpp) {$\alpha_1$};
			\node[oddbluetext, elab]  at (mpm) {$-\alpha_2$};
			\node[oddbluetext, elab]  at (mmp) {$-\alpha_3$};
			\node[oddbluetext, nelab, xshift=-6pt] at (mmm) {$-\theta$};
			
			% ------- legend ---------------------------------------------------
			\node[anchor=west, align=left, font=\small]
			at (canvas cs:x=3.6cm,y=1.1cm)
			{{\color{oddtext}\textbf{odd simple roots}}\\[1pt]
				{\color{oddtext}$\alpha_1,\alpha_2,\alpha_3$}\\[5pt]
				{\color{oddbluetext}\textbf{remaining odd roots}}\\[1pt]
				{\color{oddbluetext}$\pm\theta,\ -\alpha_1,\ -\alpha_2,\ -\alpha_3$}\\[5pt]
				{\color{eventext}\textbf{even roots}}\\[1pt]
				{\color{eventext}$\pm2\varepsilon_1,\ \pm2\varepsilon_2,\ \pm2\varepsilon_3$}\\[5pt]
				$\theta=\varepsilon_1+\varepsilon_2+\varepsilon_3$\\
				$\phantom{\theta}=\alpha_1+\alpha_2+\alpha_3$\\[2pt]
				$\alpha_0=\delta-\theta$};
		\end{tikzpicture}
		\caption{The root system of $\Ddeg$. The eight odd roots are the vertices of a cube, and the six even roots $\pm2\varepsilon_i$ lie on the coordinate axes. The three odd simple roots are marked in yellow. The vertex $-\theta$ is the finite part of the simple affine root $\alpha_0=\delta-\theta$ of $\Daff$.}
		\label{fig:RootSystem}
	\end{figure}
	Consequently $2\varepsilon_1=\alpha_2+\alpha_3$, $2\varepsilon_2=\alpha_1+\alpha_3$ and $2\varepsilon_3=\alpha_1+\alpha_2$.
	
	We fix once and for all the following notation for the formal exponents of the weights. For $\gamma \in \fh^*\oplus\mathbb{Q}\delta$ we write $X^{\gamma}$ and set
	\[X_i:=X^{\varepsilon_i},\quad i=1,2,3,\qquad q:=X^{\delta},\]
	where $\delta$ records the degree in $z$, see Subsection~\ref{ssec:Affinization}. Thus
	\[X^{\alpha_1}=X_1^{-1}X_2X_3,\quad X^{\alpha_2}=X_1X_2^{-1}X_3,\quad X^{\alpha_3}=X_1X_2X_3^{-1},\]
	and for $\{i,j,k\}=\{1,2,3\}$ one has
	\begin{equation*}
		X_{ij}:=X_{k}=X^{(\alpha_i+\alpha_j)/2}.
	\end{equation*}
	The variable $t$ plays a different role: it does not come from the weight lattice but records the $\mathbb{Z}$-grading of $\Ddeg$ described above, i.e.\ the odd degree. In particular $X^{\gamma}$ never contains $t$, and the factors $t$ appearing in the characters below are written separately.
	
	For the Lie algebra $\fsl_2$ we use the standard basis $\{e,h,f\}$.
	Consider the
	basis of the two-dimensional irreducible representation $V_{\varepsilon}=\Span{v_+,v_-}$ of $\fsl_2$ such that $v_+$ is the highest weight vector, $v_-$ is the lowest weight vector, $e v_-=v_+$ and $f v_+ =v_-$. We use a lower index to indicate the $\fsl_2$-summand, so that the basis of $\Ddeg_{\overline 0}$ is $\{e_i, f_i, h_i\}$, $i=1,2,3$.
	We choose the even root vectors as
	\[e_{2 \varepsilon_i}:=e_{i},~ e_{-2 \varepsilon_i}:=f_{i},\quad i=1,2,3,\]
	and the odd root vectors as $e_{\pm\varepsilon_1\pm\varepsilon_2\pm\varepsilon_3}:=v_\pm \otimes v_\pm \otimes v_\pm$, the signs matched factorwise, so that
	\[e_{\alpha_1}=v_- \otimes v_+ \otimes v_+, ~e_{\alpha_2}=v_+ \otimes v_- \otimes v_+, ~e_{\alpha_3}=v_+ \otimes v_+ \otimes v_-,~e_{\theta}=v_+ \otimes v_+ \otimes v_+.\]
	
	We denote by $\fb$ the Borel subalgebra spanned by the Cartan subalgebra and the positive root vectors, i.e.\
	\[\fb =\fh \oplus \Span{e_{\alpha_1}, e_{\alpha_2}, e_{\alpha_3},\ 
		e_{2\varepsilon_1}, e_{2\varepsilon_2}, e_{2\varepsilon_3}, e_{\theta}}.\]
	
	Following the convention of Definition~\ref{def:WeylGroupConvention}, the Weyl group of $\Ddeg$ is the Weyl group of $\fsl_2\oplus\fsl_2\oplus \fsl_2$, i.e.\ $W \simeq (\mathbb{Z}/2\mathbb{Z})^3$ generated by the reflections $s_1,s_2,s_3$, where $s_i(\varepsilon_i)=- \varepsilon_i$ and $s_i(\varepsilon_j)=\varepsilon_j$ for $j \neq i$.
	
	\subsection{The loop superalgebra}\label{ssec:Affinization}
	The Lie superalgebra we work with throughout is the loop superalgebra of $\Ddeg$.
	
	Let $z$ be a variable and let $\Bbbk[z,z^{-1}]$ be the algebra of Laurent polynomials in this variable. We set
	\[\Daff := \Ddeg \otimes \Bbbk[z,z^{-1}].\]
	The odd part of this superalgebra is an abelian ideal and its even part is isomorphic to $(\fsl_2\oplus \fsl_2 \oplus \fsl_2)\otimes \Bbbk[z,z^{-1}]$.
	
	\begin{remark}
		We do not need the affine Kac--Moody superalgebra attached to $\Ddeg$, with its central element $K$ and derivation $d=z\partial_z$: the loop superalgebra $\Daff$ suffices. In terms of the affine algebra, our modules are the modules of level zero, on which $K$ acts by $0$, and their grading by the degree in $z$ is the grading recorded by $d$. The superscript $(1)$ does not indicate a central extension.
	\end{remark}
	
	The superalgebra $\Daff$ is $\mathbb{Z}$-graded by the degree in $z$, and $\fh$ acts on it semisimply by the adjoint action. We record the two gradings at once. Let $\delta$ be a formal symbol and put
	\[\Daff_{\beta+r\delta}:=\Ddeg_\beta\otimes z^r,\qquad \beta\in\fh^*,\ r\in\mathbb{Z},\]
	so that $x\otimes z^r$ has weight $\beta+r\delta$ for $x\in\Ddeg_\beta$: the finite part $\beta$ is the $\fh$-weight, and the coefficient of $\delta$ is the degree in $z$. The root system is
	\[\Phi^{\af}:=\{\gamma\neq0\mid\Daff_\gamma\neq0\}=\Phi^{\af}_{\overline 0} \sqcup \Phi^{\af}_{\overline 1}.\]
	
	Explicitly,
	\[\Phi^{\af}_{\overline 0}=\{\pm 2 \varepsilon_i+r \delta \mid i=1,2,3,~r\in\mathbb{Z}\}\cup\{r\delta\mid r \in \mathbb{Z}\setminus\{0\}\},\]
	\[\Phi^{\af}_{\overline 1}=\{\pm \varepsilon_1 \pm \varepsilon_2 \pm \varepsilon_3+ r \delta \mid r \in \mathbb{Z}\}.\]
	
	All the real roots of $\Phi^{\af}$ have multiplicity one. The imaginary roots $r\delta$, $r \neq 0$, have multiplicity three: the corresponding root space is spanned by $h_1z^r, h_2z^r,h_3z^r$. This multiplicity is the source of the factor $(q;q)_\infty^{-3}$ in the character of the enveloping algebra of the radical computed in Subsection~\ref{ssec:ProjectiveModules}.
	
	For an even real root $\gamma=\pm2\varepsilon_i+r\delta$ its coroot is $h_\gamma:=[e_{\pm2\varepsilon_i}z^r,e_{\mp2\varepsilon_i}z^{-r}]=\pm h_i$, so that $\gamma(h_\gamma)=2$, and weights are paired with coroots by evaluation. In particular $\varepsilon_1,\varepsilon_2,\varepsilon_3$ is the basis of $\fh^*$ dual to $h_1,h_2,h_3$.
	
	We define the Iwahori subalgebra $\Iw \subset \Daff$ by
	\[\Iw:=\Ddeg \otimes z\Bbbk[z]\oplus\fb.\]
	
	We denote by $\Phi^{\af}_+$ the positive part of the root system $\Phi^{\af}$ with respect to $\Iw$. That is, $\gamma \in \Phi^{\af}_+$ if and only if the root space $\Daff_\gamma$ is contained in $\Iw$.
	
	Define the simple affine root
	\[\alpha_0=-\varepsilon_1 - \varepsilon_2 -\varepsilon_3 + \delta\]
	and the corresponding root vector
	\[e_{\alpha_0}=v_-\otimes v_- \otimes v_- \otimes z.\]
	Since $\theta$ is the highest root of the finite root system, we obtain the following description
	of $\Phi^{\af}_+$.
	
	\begin{lemma}
		\[\Phi^{\af}_+=\left(\sum_{i=0}^3\mathbb{Z}_{\geq 0}\alpha_i\right)\cap\Phi^{\af}.\]
	\end{lemma}
	
	For any subset $J \subset \{0,1,2,3\}$ we construct a parabolic subalgebra in two steps. First we define the corresponding subset of the root system:
	
	\begin{definition}\label{def:parabolicRoots}
		\[\Phi^{\af}_{J}:=\Phi^{\af}_+\cup\left(\left(\sum_{i\in J}\mathbb{Z}_{\geq 0}(-\alpha_i)\right)\cap\Phi^{\af}\right).\]
	\end{definition}
	
	The parabolic subalgebra itself is then defined as follows.
	
	\begin{definition}\label{def:parabolicSubalgebras}
		\[\fp_J:=\fh \oplus \bigoplus_{\gamma \in \Phi^{\af}_{J}}\Daff_{\gamma}.\]
	\end{definition}
	
	Since every root lies in $\Phi^{\af}_+$ or in $-\Phi^{\af}_+$, the set $\Phi^{\af}_J$ consists of the roots $\sum_m a_m\alpha_m$ with $a_m\geq 0$ for $m\notin J$. This set is closed under addition within $\Phi^{\af}$, and a sum equal to $0$ lands in $\fh$, so $\fp_J$ is a subalgebra.
	
	By definition we have
	\begin{equation*}
		\fp_{\emptyset} =\Iw,~\fp_{123}=\Ddeg \otimes \Bbbk[z],~ \fp_{J_1} \subset \fp_{J_2}~\text{if and only if } J_1 \subset J_2.
	\end{equation*}
	For $J\neq\{0,1,2,3\}$ the reductive part $\fp_J^{\mathrm{red}}$ is spanned by $\fh$ and the root spaces of the even real roots $\gamma$ with $\pm\gamma\in\Phi^{\af}_J$, and the radical $R$ is spanned by the remaining root spaces of $\fp_J$.
	
	\begin{remark}\label{rem:ParabolicDegeneration}
		In the semisimple case a parabolic subalgebra can be defined as the subalgebra generated by the Cartan subalgebra, all positive root vectors and a chosen set of simple negative root vectors. That recipe fails in the degenerate case considered here. Our definition, however, ensures that our parabolic subalgebras are the degenerations of the corresponding parabolic subalgebras of the loop superalgebra of $D(2|1,a_1,a_2,a_3)$ for generic values of the parameters.
	\end{remark}
	
	\subsection{Diagrammatic automorphisms}\label{ssec:DiagrammaticAutomorphisms}
	The goal of this subsection is to construct a family of automorphisms of the loop superalgebra $\Daff$ preserving its Iwahori subalgebra $\Iw$. We call such automorphisms \emph{diagrammatic} to emphasize that they are analogues of those induced by symmetries of the Dynkin diagram of a semisimple Lie (super)algebra.
	
	The diagrammatic automorphisms we construct are indexed by the symmetric group $\mathfrak S_4$, acting naturally on the set $\{0,1,2,3\}$ of indices of the simple roots. We define the corresponding action on the root system first. For $\sigma \in \mathfrak S_4$ we set
	\[\pi_\sigma(\alpha_i):=\alpha_{\sigma(i)},\]
	and extend this linearly to $\fh^*\oplus\mathbb{Q}\delta$, of which $\alpha_0,\dots,\alpha_3$ form a basis. In particular, for $\gamma=\sum_{i=0}^3a_i\alpha_i \in \Phi^{\af}$,
	\begin{equation*}
		\pi_\sigma(\gamma)=\sum_{i=0}^3a_i\alpha_{\sigma(i)}.
	\end{equation*}
	
	\begin{lemma}
		For every $\gamma\in\Phi^{\af}$ one has $\pi_\sigma(\gamma)\in \Phi^{\af}$. If moreover $\gamma \in \Phi^{\af}_+$, then $\pi_\sigma(\gamma)\in \Phi^{\af}_+$.
	\end{lemma}
	\begin{proof}
		Note that $\delta=\alpha_0+\alpha_1+\alpha_2+\alpha_3$, so $\pi_\sigma(\delta)=\delta$. One can rewrite
		\[\gamma=\sum_{i=1}^3(a_i-a_0)\alpha_i+a_0 \delta,\]
		and the condition $\gamma \in \Phi^{\af}$ is equivalent to $\gamma^{\mathrm{re}}:=\sum_{i=1}^3(a_i-a_0)\alpha_i \in \Phi\cup\{0\}$, with $a_0\neq0$ if $\gamma^{\mathrm{re}}=0$. If $\gamma^{\mathrm{re}}=0$, then $\gamma=a_0\delta$ and $\pi_\sigma(\gamma)=\gamma$.
		Otherwise $\gamma^{\mathrm{re}}\in\Phi$ and $\pi_\sigma(\gamma)=\pi_\sigma(\gamma^{\mathrm{re}})+a_0\delta$, so it suffices to check that $\pi_\sigma(\gamma^{\mathrm{re}}) \in \Phi^{\af}$. This follows from the description $\Phi=\{\pm \alpha_i, \pm (\alpha_0-\delta),\pm (\alpha_i +\alpha_j)\mid i,j\in\{1,2,3\},\ i\neq j\}$, which gives $\pi_\sigma(\Phi) \subset \Phi^{\af}$.
		
		The second claim of the lemma is now immediate from the definition.
	\end{proof}
	
	\begin{proposition}
		The action of $\mathfrak{S}_4$ on $\Phi^{\af}$ extends to an action on the Lie superalgebra $\Daff$ by automorphisms $\pi_\sigma$ such that
		$\pi_\sigma(e_{\alpha_i})=e_{\alpha_{\sigma(i)}}$, $i=0,1,2,3$.
	\end{proposition}
	\begin{proof}
		The odd part of $\Daff$ is an abelian ideal. Hence an automorphism $\varphi$ of the even part $(\fsl_2\oplus\fsl_2\oplus\fsl_2)\otimes\Bbbk[z,z^{-1}]$ together with a linear automorphism $\psi$ of the odd part $V_\varepsilon\otimes V_\varepsilon\otimes V_\varepsilon\otimes\Bbbk[z,z^{-1}]$ such that $\psi(xv)=\varphi(x)\psi(v)$ is an automorphism of $\Daff$. If $\sigma(0)=0$, let $\varphi$ and $\psi$ permute the three factors by $\sigma$. For $\sigma=(0,1)$ let $\varphi$ be the identity on the first factor and, on the second and the third, the exchange of the two factors followed by
		\[e\mapsto fz,\qquad f\mapsto ez^{-1},\qquad h\mapsto -h\]
		in each of them. Let $\psi$ be the identity on the first tensor factor and, on the other two, the exchange followed by $v_\pm\mapsto v_\mp z^{\pm1/2}$ in each of them. The powers $z^{\pm1/2}$ occur in pairs, so $\psi$ preserves $V_\varepsilon\otimes V_\varepsilon\otimes V_\varepsilon\otimes\Bbbk[z,z^{-1}]$. In both cases $\pi_\sigma(e_{\alpha_i})=e_{\alpha_{\sigma(i)}}$. The transpositions $(0,1)$, $(1,2)$, $(2,3)$ generate $\mathfrak S_4$, and the relations of $\mathfrak S_4$ are checked directly.
	\end{proof}
	
	\begin{example}
		For $\pi=\pi_{(0,1)}$ the construction gives
		\[\pi(e_{2\varepsilon_2}z^k)=e_{-2\varepsilon_3}z^{k+1},\qquad \pi(e_{-2\varepsilon_2}z^k)=e_{2\varepsilon_3}z^{k-1},\qquad \pi(h_2)=-h_3,\]
		\[\pi(e_{\alpha_1})=v_-\otimes v_-\otimes v_-\,z=e_{\alpha_0},\qquad \pi(e_{\theta})=v_+\otimes v_-\otimes v_-\,z=e_{-\alpha_1}z,\]
		and $\pi$ fixes $e_{\pm2\varepsilon_1}$, $h_1$, $e_{\pm\alpha_2}$ and $e_{\pm\alpha_3}$. For instance, the bracket $[e_{2\varepsilon_2}z^k,e_{-2\varepsilon_2}z^{-k}]=h_2$ is sent to $[e_{-2\varepsilon_3}z^{k+1},e_{2\varepsilon_3}z^{-k-1}]=-h_3$. On weights $\pi$ acts by
		\[\varepsilon_1\mapsto\varepsilon_1,\qquad \varepsilon_2\mapsto\tfrac{\delta}{2}-\varepsilon_3,\qquad \varepsilon_3\mapsto\tfrac{\delta}{2}-\varepsilon_2.\]
		This is where the half-integral $z$-degrees come from.
	\end{example}
	
	\begin{corollary}\label{cor:ParabolicStable}
		Consider the parabolic subalgebra $\fp_{ij}$ for a two-element subset $\{i,j\}\subset \{0,1,2,3\}$ and denote $\{k,l\}:=\{0,1,2,3\} \setminus \{i,j\}$. Then $\fp_{ij}$ is stable under $\pi_{(i,j)(k,l)}$.
	\end{corollary}
	
	\section{Categories of representations and Grothendieck groups}\label{sec:Categories}
	In this section we first recall the standard homological notions which we use, referring to~\cite{weibel1994introduction} for the details, then the definition of the Grothendieck group and of the maps induced by functors on it. After that we introduce the categories of representations we work with and compute their Grothendieck groups in terms of characters.
	
	\subsection{Recollections from homological algebra}
	Let $\mathcal{O}$ be an abelian category.
	A \emph{chain complex} \( M^\bullet \) in the abelian category \( \mathcal{O} \) is a family of objects \( \{M^n\}_{n \in \mathbb{Z}} \) in \( \mathcal{O} \) together with morphisms (differentials)
	\[
	d_M^n: M^n \longrightarrow M^{n+1},~
	d_M^{n+1} \circ d_M^n = 0
	\]
	for every \( n \in \mathbb{Z} \).
	
	A \emph{morphism} of chain complexes \( f: M^\bullet \to N^\bullet \) is a family of morphisms \( f^n: M^n \to N^n \) such that
	\[
	f^{n+1} \circ d_M^n = d_N^n \circ f^n
	\]
	for all \( n \).

	Let \( \Kom(\mathcal{O}) \) be the category of chain complexes, whose objects are chain complexes and whose morphisms are chain maps.	
	For a complex $M^\bullet$ the $i$-th cohomology object is
	\[H^{i}(M^\bullet)=\ker d_M^{i}/\operatorname{im}d_{M}^{i-1}.\]
	A chain map is a \emph{quasi-isomorphism} if it induces isomorphisms on all cohomology objects.
	
	The \emph{derived category} \( D(\mathcal{O}) \) is the localization of \( \Kom(\mathcal{O}) \) with respect to all quasi-isomorphisms:
	\[
	D(\mathcal{O}) := \Kom(\mathcal{O})[\mathcal{Q}^{-1}],
	\]
	where \( \mathcal{Q} \) is the class of quasi-isomorphisms.
	
	The subcategory of the derived category which consists of the objects that have right-bounded representatives is denoted by $D^-(\mathcal{O})$. In a similar way $D^+(\mathcal{O})$ consists of the classes of left-bounded complexes, and $D^b(\mathcal{O})$ of the classes of bounded complexes.
	
	Let \(\psi: \mathcal{O} \to \mathcal{A}\) be a right exact functor between abelian categories, where \(\mathcal{O}\) has enough projectives.
	The left derived functor of \(\psi\) is the functor
	\[
	\mathbf{L}\psi: D^-(\mathcal{O}) \longrightarrow D^-(\mathcal{A})
	\]
	defined as follows.	
	For any complex \(X^\bullet \in D^-(\mathcal{O})\), choose a projective resolution
	\[
	P^\bullet \xrightarrow{\sim} X^\bullet,
	\]
	where \(P^\bullet\) is a right-bounded complex of projective objects in \(\mathcal{O}\) quasi-isomorphic to \(X^\bullet\).
	Then set
	\[
	\mathbf{L}\psi(X^\bullet) := \psi(P^\bullet) \in D^-(\mathcal{A}),
	\]
	where \(\psi(P^\bullet)\) is the complex obtained by applying \(\psi\) object-wise to \(P^\bullet\).
	For a morphism in \(D^-(\mathcal{O})\), \(\mathbf{L}\psi\) is defined by lifting it to projective resolutions. This is well defined up to unique isomorphism.
	
	More generally, the projective objects may be replaced by any class $\mathcal{P}$ of objects of $\mathcal{O}$ adapted to $\psi$. This means that every object of $\mathcal{O}$ is a quotient of an object of $\mathcal{P}$, that $\mathcal{P}$ is closed under finite direct sums and under kernels of epimorphisms between its objects, and that $\psi$ takes short exact sequences of objects of $\mathcal{P}$ to short exact sequences. Then every $X^\bullet \in D^-(\mathcal{O})$ admits a quasi-isomorphism $P^\bullet \rightarrow X^\bullet$ from a right-bounded complex $P^\bullet$ of objects of $\mathcal{P}$, and $\mathbf{L}\psi(X^\bullet):=\psi(P^\bullet)$ is well defined, see~\cite{KashiwaraSchapira-2006}. For the functors of Section~\ref{sec:Functors} we use adapted classes of finite-dimensional modules.
	
	Dually, let $\psi$ be left exact and let $\mathcal{B}$ be a class adapted to $\psi$ in the dual sense: every object embeds into an object of $\mathcal{B}$, the class $\mathcal{B}$ is closed under finite direct sums and under cokernels of monomorphisms between its objects, and $\psi$ takes short exact sequences of objects of $\mathcal{B}$ to short exact sequences. Then the right derived functor $\mathbf{R}\psi: D^+(\mathcal{O}) \rightarrow D^+(\mathcal{A})$ is given by $\mathbf{R}\psi(X^\bullet)=\psi(B^\bullet)$ for a quasi-isomorphism $X^\bullet \rightarrow B^\bullet$ into a left-bounded complex of objects of $\mathcal{B}$.
	
	For an object \(X \in \mathcal{O}\), viewed as a complex concentrated in degree $0$, we put
	\[
	\mathbf{L}^{(n)}\psi(X):=H^{-n}(\mathbf{L}\psi(X)),\qquad \mathbf{R}^{(n)}\psi(X):=H^{n}(\mathbf{R}\psi(X)).
	\]
	Since $\psi$ is right exact, $\mathbf{L}^{(0)}\psi=\psi$, and dually $\mathbf{R}^{(0)}\psi=\psi$ for a left exact $\psi$. With enough projectives, resp.\ injectives, these are the classical derived functors. An object $A$ with $\mathbf{L}^{(n)}\psi(A)=0$ for $n>0$ is called $\psi$-acyclic. Dually, for a left exact $\psi$ an object $B$ with $\mathbf{R}^{(n)}\psi(B)=0$ for $n>0$ is called $\psi$-acyclic. Objects of an adapted class are $\psi$-acyclic, and the $\psi$-acyclic objects themselves form a class adapted to $\psi$. In other words, all adapted classes define the same functor $\mathbf{L}\psi$, resp.\ $\mathbf{R}\psi$.
	Recall the following definition of adjoint functors.
	\begin{definition}\label{def:Adjunction}    
		An adjunction between \(\mathcal{O}\) and \(\mathcal{C}\) is given by functors 
		\(\psi: \mathcal{O} \to \mathcal{C}\) and \(\phi: \mathcal{C} \to \mathcal{O}\) 
		together with a bijection
		\[
		\Hom_{\mathcal{C}}(\psi(M), N) \simeq \Hom_{\mathcal{O}}(M, \phi(N))
		\]
		for all objects \(M \in \mathcal{O}\) and \(N \in \mathcal{C}\).
		The bijection is required to be natural in both variables \(M\) and \(N\).
	\end{definition}
	
	\subsection{Grothendieck groups}\label{ssec:GrothendieckGroups}
	Consider an abelian category $\mathcal{O}$. Recall that the Grothendieck group $K_0(\mathcal{O})$ is the abelian group generated by the symbols $[M]$, $M \in \mathcal{O}$, subject to the relation $[M]=[L]+[N]$ for every short exact sequence
	\begin{equation}\label{eq:ShortExactSequence}
		0 \rightarrow L \rightarrow M \rightarrow N \rightarrow 0.
	\end{equation}
	In words, the class of an object is equal to the sum of the classes of a subobject and of the corresponding quotient object. 
	
	Recall also that the natural map $K_0(\mathcal{O}) \rightarrow K_0(D^b(\mathcal{O}))$, $[M]\mapsto [M]$, is an isomorphism, the inverse being given by the Euler characteristic $[M^\bullet]\mapsto \sum_i(-1)^i[H^i(M^\bullet)]$. We use this identification without further mention.
	
	Now let $\eta$ be an exact functor $\mathcal{O} \rightarrow \mathcal{C}$. For any exact sequence \eqref{eq:ShortExactSequence} it gives the short exact sequence
	\[
	0 \rightarrow \eta(L) \rightarrow \eta(M) \rightarrow \eta(N) \rightarrow 0.\]
	Thus 
	\[[\eta(M)]=[\eta(L)]+[\eta(N)].\]
	
	Therefore each exact functor defines a map on the Grothendieck groups $[\eta]:K_0(\mathcal{O})\rightarrow K_0(\mathcal{C})$, $[M] \mapsto [\eta(M)]$.
	
	More generally, assume that $\eta$ is a right exact functor of finite homological dimension $k$, i.e.\ $\mathbf{L}^{(i)}\eta=0$ for $i>k$. Then for any short exact sequence \eqref{eq:ShortExactSequence} we get
	\[0\rightarrow  \mathbf{L}^{(k)}\eta(L) \rightarrow \mathbf{L}^{(k)}\eta(M) \rightarrow \mathbf{L}^{(k)}\eta(N) \rightarrow  \dots\rightarrow \mathbf{L}^{(1)}\eta(L) \rightarrow \mathbf{L}^{(1)}\eta(M) \rightarrow \mathbf{L}^{(1)}\eta(N)  \rightarrow \eta(L) \rightarrow \eta(M) \rightarrow \eta(N) \rightarrow 0.\]
	
	Then the following relation holds in the Grothendieck group:
	\[\sum_{i=0}^k (-1)^i[\mathbf{L}^{(i)}\eta(M)]=\sum_{i=0}^k (-1)^i [\mathbf{L}^{(i)}\eta(L)]+\sum_{i=0}^k (-1)^i [\mathbf{L}^{(i)}\eta(N)].\]
	Therefore each such functor defines a map on the Grothendieck groups
	\begin{equation*}
		[\eta]:K_0(\mathcal{O}) \rightarrow K_0(\mathcal{C}),\qquad [M] \mapsto \sum_{i=0}^k (-1)^i[\mathbf{L}^{(i)}\eta(M)].
	\end{equation*}
	
	In the same way a left exact functor $\eta$ with $\mathbf{R}^{(i)}\eta=0$ for $i>k$ defines the map $[M]\mapsto\sum_{i=0}^k(-1)^i[\mathbf{R}^{(i)}\eta(M)]$.
	
	\subsection{Categories of representations}\label{ssec:CategoriesOfRepresentations}
	Let $\fg$ be a Lie superalgebra graded by the $z$-degree and by the odd degree, with a fixed decomposition \eqref{eq:ReductivePart}, and let $\fh\subset\fg^{\mathrm{red}}$ be a Cartan subalgebra with weight lattice $P\subset\fh^*$.
	
	\begin{definition}\label{def:CategoryO}
		We denote by $\Ofd(\fg)$ the category of finite-dimensional $\fg$-modules $M$ equipped with a decomposition
		\[M=\bigoplus_{\lambda\in P,\;k\in\frac12\mathbb{Z},\;r\in\mathbb{Z}}M_{\lambda,k,r}\]
		such that $\fh$ acts on $M_{\lambda,k,r}$ by $\lambda$, an element of $\fg$ of weight $\gamma$, $z$-degree $l$ and odd degree $s$ maps $M_{\lambda,k,r}$ to $M_{\lambda+\gamma,k+l,r+s}$, and the parity of $M_{\lambda,k,r}$ is that of $r$. Morphisms are the $\fg$-linear maps preserving the decomposition.
	\end{definition}
	
	The radical $R$ acts nilpotently on each object of $\Ofd(\fg)$, hence annihilates each irreducible object. Indeed, one checks on the root vectors that for suitable constants $a_1,a_2,a_3\geq 0$ the degree $l+r+\sum_{m=1}^3a_m\lambda_m$ of an element of weight $\lambda$, $z$-degree $l$ and odd degree $r$ vanishes on $\fg^{\mathrm{red}}$ and is positive on $R$, for instance $a_1=a_2=a_3=0$ for $\fp_{123}$ and $a_i=a_j=\frac14$, $a_k=0$ for $\fp_{ij}$, $\{i,j,k\}=\{1,2,3\}$. Here $\lambda=\sum_m\lambda_m\varepsilon_m$. On a finite-dimensional module $R$ raises the corresponding degree, so it acts nilpotently. Therefore
	\[K_0(\Ofd(\fg))\simeq K_0(\Ofd(\fg^{\mathrm{red}})).\]
	
	For $V\in\Ofd(\fg)$ the tensor product functor $\bullet \otimes V$ is an exact endofunctor of $\Ofd(\fg)$.
	
	\begin{definition}\label{def:ShiftDefinition}
		For $M \in \Ofd(\fg)$, $l\in\frac12\mathbb{Z}$ and $r \in \mathbb{Z}$ we denote by $M\sh{l}{r}$ the module $M$ with the decomposition $\bigl(M\sh{l}{r}\bigr)_{\lambda,k,s}:=M_{\lambda,k-l,s-r}$, that is, with the $z$-grading shifted by $l$ and the odd grading shifted by $r$.
	\end{definition}
	The homological shift of a complex is $M^\bullet[n]$, $\bigl(M^\bullet[n]\bigr)^i=M^{i+n}$. For a module $M$ the complex $\mathbf{L}\psi(M)$ has cohomology only in non-positive degrees, and $\mathbf{R}\psi(M)$ only in non-negative degrees. Angle brackets refer to the two internal gradings, square brackets to the homological one. In the Grothendieck group $\bigl[M^\bullet[1]\bigr]=-[M^\bullet].$
	
	\subsection{Characters}\label{ssec:Characters}
	Let us now compute the Grothendieck groups of the categories $\Ofd(\fp_J)$ in terms of characters.
	
	Recall that the weight lattice is $\mathbb{Z}^3$ with the basis $\varepsilon_1,\varepsilon_2,\varepsilon_3$. By Definition~\ref{def:CategoryO} an object $M \in \Ofd(\fp_J)$ is equipped with a decomposition
	\begin{equation*}
		M=\bigoplus_{\lambda \in \mathbb{Z}^3,k\in \frac{1}{2}\mathbb{Z},r\in \mathbb{Z}}M_{\lambda,k,r}
	\end{equation*}
	and $h_i m=\lambda_i m$ for $m \in M_{\lambda,k,r}$.
	
	Then the character of the module $M$ is defined to be the graded superdimension of $M$, that is, the variable recording the odd grading is $-t$:
	\begin{equation}\label{eq:CharacterDefinition}
		\ch(M):=\sum_{\lambda \in \mathbb{Z}^3,k\in \frac{1}{2}\mathbb{Z},r\in \mathbb{Z}} \dim(M_{\lambda,k,r})X_1^{\lambda_1}X_2^{\lambda_2}X_3^{\lambda_3}q^k(-t)^r.
	\end{equation}
	For $M\in\Ofd(\fp_J)$ this is a Laurent polynomial. The same formula defines the character of any graded vector space with finite-dimensional components $M_{\lambda,k,r}$, such as the algebras in the example below.
	
	\begin{remark}\label{rem:SignConvention}
		The sign in \eqref{eq:CharacterDefinition} is a normalization of the variable $t$ and nothing more: replacing $t$ by $-t$ turns $\ch$ into the naive graded dimension. We use this normalization for two reasons. The $\mathbb{Z}$-grading by $r$ refines the parity, since the component $M_{\lambda,k,r}$ is even or odd according to the parity of $r$, so $\ch$ is the supercharacter. Moreover, with it the exterior algebra on an odd one-dimensional space of weight $\beta$ has character $1-tX^{\beta}$. This is the form in which all the products below appear, see \eqref{eq:FirstInductionCharacter}, the character of $\U(R)$ in Subsection~\ref{ssec:ProjectiveModules} and the numerators \eqref{eq:Cab} of the operators of~\cite{ArthamonovShakirov-2019}.
	\end{remark}

	Note that
	\begin{equation}
		\ch\bigl(M\sh{l}{r}\bigr)=q^l(-t)^r\ch(M).
	\end{equation}
	Moreover, for any short exact sequence \eqref{eq:ShortExactSequence} one has
	\[\ch(M)=\ch(L)+\ch(N).\]
	
	Therefore there is a natural map of abelian groups
	\begin{equation*}
		\iota_{J}:K_0(\Ofd(\fp_{J})) \rightarrow \mathbb{Z}[X_1^{\pm 1},X_2^{\pm 1},X_3^{\pm 1}][t^{\pm 1},q^{\pm\frac{1}{2}}].
	\end{equation*}
	Every object of $\Ofd(\fp_J)$ has finite length, hence $K_0(\Ofd(\fp_J))$ is the free abelian group on the classes of the irreducible objects.

	Consider now the case $J=\{1,2,3\}$.
	Up to the shifts $\sh{l}{r}$ the irreducible objects of $\Ofd(\fp_{123})$ have the form $V_{123}(\mu_1,\mu_2,\mu_3):=V_{\mu_1 \varepsilon_1} \otimes V_{\mu_2 \varepsilon_2} \otimes V_{\mu_3 \varepsilon_3}$, where $V_{\mu_i \varepsilon_i}$ is the irreducible representation of the $i$-th copy of the Lie algebra $\fsl_2$, $\mu_i\geq 0$, and $V_{123}(\mu)$ is concentrated in $z$-degree $0$ and odd degree $0$. Its character is
	\begin{equation}\label{eq:character3Irreducible}
		\ch(V_{123}(\mu_1,\mu_2,\mu_3))= \prod_{i=1}^3 \frac{X_i^{\mu_i}-X_i^{-\mu_i-2}}{1-X_i^{-2}}.
	\end{equation} 
	Recall the reflections $s_i, i=1,2,3$, acting in the following way
	\[s_1(X_1^{\mu_1}X_2^{\mu_2}X_3^{\mu_3}q^kt^r)=X_1^{-\mu_1}X_2^{\mu_2}X_3^{\mu_3}q^kt^r.\]
	Then the following lemma is clear.
	\begin{lemma}\label{lem:K0}
		The map $\iota_{123}$ is injective, and
		\[\iota_{123}(K_0(\Ofd(\fp_{123})))=\mathbb{Z}[X_1^{\pm 1},X_2^{\pm 1},X_3^{\pm 1}]^{s_1,s_2,s_3}[t^{\pm 1},q^{\pm\frac{1}{2}}].\]
	\end{lemma}
	
	In the same way consider the case $J=\{1,2\}$. In this case
	\[\fp_{12}^{\mathrm{red}}=\Span{h_1,h_2} \oplus \fsl_2^{(3)},\]
	where $\fsl_2^{(3)}$ is the third copy of $\fsl_2$.
	Up to the shifts $\sh{l}{r}$ the irreducible objects of $\Ofd(\fp_{12})$ have the form $V_{12}(\mu_1,\mu_2,\mu_3):=L_{\mu_1 \varepsilon_1}\otimes L_{\mu_2 \varepsilon_2}\otimes V_{\mu_3 \varepsilon_3}$, where $L_{\mu_i \varepsilon_i}$ is the one-dimensional representation of the Lie algebra $\Span{h_i}$, $\mu_1,\mu_2 \in \mathbb{Z}$, $\mu_3\geq 0$, and $V_{12}(\mu)$ is again concentrated in $z$-degree $0$ and odd degree $0$. Its character is
	\begin{equation*}
		\ch(V_{12}(\mu_1,\mu_2,\mu_3))= X_1^{\mu_1}X_2^{\mu_2} \frac{X_3^{\mu_3}-X_3^{-\mu_3-2}}{1-X_3^{-2}}.
	\end{equation*} 
	
	The same computations hold for $\fp_{13}$ and $\fp_{23}$. Thus we have
	\begin{lemma}\label{lem:K0ik}
		Let $\{i,k\}=\{1,2,3\}\setminus \{j\}$. The map $\iota_{ik}$ is injective, and
		\[\iota_{ik}(K_0(\Ofd(\fp_{ik})))=\mathbb{Z}[X_1^{\pm 1},X_2^{\pm 1},X_3^{\pm 1}]^{s_j}[t^{\pm 1},q^{\pm\frac{1}{2}}].\]
	\end{lemma}

	\begin{example} We list the characters of the algebras we deal with.
		\[\ch(\Iw)=3+X_1^2+X_2^2+X_3^2-tX_1^{-1}X_2X_3-tX_1X_2^{-1}X_3-tX_1X_2X_3^{-1}-tX_1X_2X_3+\frac{q}{1-q}\ch \Ddeg,\]
		\begin{equation*}
			\ch(\fp_{12})=\ch(\Iw)-tX_1^{-1}X_2X_3^{-1}-tX_1X_2^{-1}X_3^{-1}+X_3^{-2},
		\end{equation*}
		\begin{equation*}
			\ch(\fp_{123})=\frac{1}{1-q}\ch \Ddeg.
		\end{equation*}
		
		Here $\ch \Ddeg$ is the character of the finite-dimensional superalgebra and
		\[\ch \Ddeg=\sum_{i=1}^3 (X_i^{2}+1+X_i^{-2})-t(X_1+X_1^{-1})(X_2+X_2^{-1})(X_3+X_3^{-1}).\]
	\end{example}
	
	\section{Functors on the categories of representations}\label{sec:Functors} 
	The goal of this section is to study the functors between categories of representations of the parabolic subalgebras of $\Daff$. The main results of this paper mean that these functors behave in the same way as the functors between certain parabolic subalgebras of the Lie superalgebra $\fg \otimes \Bbbk[\xi]$, where $\xi$ is an odd variable and $\fg$ is a Kac--Moody Lie algebra, see~\cite{KKhM}. This similarity gives us a way to see why the algebra of the decategorifications of these functors behaves similarly to the classical double affine Hecke algebra~\cite{Cherednik-2005}. For the general theory of Zuckerman functors, see for example~\cite{Vogan1981}.
	
	All the functors we consider relate a pair of Lie superalgebras $\fg_1 \subset \fg_2$. Throughout the paper such a pair is assumed to satisfy the following assumption.
	
	\begin{assumption}\label{ass}
		Let $\fb_2$ be a Borel subalgebra of $\fg_2^{\mathrm{red}}$, let $\fb_2^-$ be the opposite Borel subalgebra and let $\fh \subset \fb_2$ be the corresponding Cartan subalgebra. We assume that
		\begin{itemize}
			\item $\fb_2 \subset \fg_1$,
			\item $(\fg_1)_{\overline 0}+ \fb_2^-=(\fg_2)_{\overline 0}$,
			\item $\dim(\fg_2/\fg_1) < \infty$,
			\item $\fg_1$ is a subalgebra of $\fg_2$ graded by the $z$-degree and the odd degree, and $\fh$ is a Cartan subalgebra of $\fg_1^{\mathrm{red}}$, with the same weight lattice $P\subset\fh^*$ for $\fg_1$ and for $\fg_2$.
		\end{itemize}
	\end{assumption}
	
	\subsection{Restriction and induction functors}
	Consider a pair of Lie superalgebras $\fg_1 \subset \fg_2$.
	
	\begin{definition}
		The restriction functor $\Res_{\fg_1}^{\fg_2}$ sends a $\fg_2$-module $M$ to the same vector space regarded as a $\fg_1$-module, and a morphism of $\fg_2$-modules to the same map of vector spaces.
	\end{definition}
	
	This functor is exact.
	
	\begin{definition}
		For a \(\fg_1\)-module \(V\) the {induced} and the {coinduced} modules are
		\[
		\Ind_{\fg_1}^{\fg_2}(V) = \U(\fg_2) \otimes_{\U(\fg_1)} V,\qquad
		\coInd_{\fg_1}^{\fg_2}(V) = \Hom_{\U(\fg_1)}(\U(\fg_2),V).
		\]
	\end{definition}
	
	Recall that the reductive part $\fg_2^{\mathrm{red}}$ of $\fg_2$ is contained in the even subalgebra $(\fg_2)_{\overline 0}$. A $\fg_2$-module is called $\fg_2^{\mathrm{red}}$-integrable, or integrable for short, if each of its vectors lies in a finite-dimensional $\fg_2^{\mathrm{red}}$-submodule. Finite-dimensional modules are integrable, and submodules, quotients and sums of integrable modules are integrable. Let $M$ be a $\fg_2$-module.
	\begin{definition}
		We denote by $M_{\mathrm{int}}$ the maximal $\fg_2^{\mathrm{red}}$-integrable quotient of $M$, that is, the quotient of $M$ by the smallest submodule with a $\fg_2^{\mathrm{red}}$-integrable quotient, and by $M^{\mathrm{int}}$ the maximal $\fg_2^{\mathrm{red}}$-integrable submodule of $M$, whenever they exist.
	\end{definition}
	
	\begin{definition}
		The integrable induction and coinduction functors are
		\[
		\widetilde \Ind_{\fg_1}^{\fg_{2}}(V) = \bigl(\U(\fg_2) \otimes_{\U(\fg_1)} V\bigr)_{\mathrm{int}},\qquad
		\widetilde \coInd_{\fg_1}^{\fg_{2}}(V) = \bigl(\Hom_{\U(\fg_1)}(\U(\fg_2),V)\bigr)^{\mathrm{int}}.
		\]
	\end{definition}
	
	\begin{lemma}\label{lem:IntegrableWellDefined}
		The functor $\Res_{\fg_1}^{\fg_2}$ maps $\Ofd(\fg_2)$ to $\Ofd(\fg_1)$. For $M\in\Ofd(\fg_1)$ the maximal integrable quotient of $\Ind_{\fg_1}^{\fg_2}(M)$ and the maximal integrable submodule of $\coInd_{\fg_1}^{\fg_2}(M)$ exist and are finite-dimensional, and with the induced decompositions they are objects of $\Ofd(\fg_2)$. So $\widetilde\Ind_{\fg_1}^{\fg_2}$ and $\widetilde\coInd_{\fg_1}^{\fg_2}$ are functors $\Ofd(\fg_1)\rightarrow\Ofd(\fg_2)$.
	\end{lemma}
	\begin{proof}
		The first claim follows from Assumption~\ref{ass}. Write $\fg_2=\fg_1\oplus\fn$ with $\fn$ spanned by homogeneous weight vectors. Then $\Ind_{\fg_1}^{\fg_2}(M)\simeq S(\fn)\otimes M$ as $\fh$-modules, and by the same assumption the even part of $\fn$ has negative weights. So the weights of $\Ind_{\fg_1}^{\fg_2}(M)$ are bounded from above, its weight spaces are finite-dimensional, and only finitely many of its weights have their whole Weyl group orbit among its weights. The quotient by the submodule generated by the remaining weight spaces is the maximal integrable quotient. The argument for $\coInd_{\fg_1}^{\fg_2}(M)$ is dual.
	\end{proof}
	
	\begin{proposition}\label{prop:Adjoint}
     The following holds:
		\begin{itemize}
			\item $\widetilde\Ind_{\fg_1}^{\fg_2}:\Ofd(\fg_1)\rightarrow \Ofd(\fg_2)$ is left adjoint to $\Res_{\fg_1}^{\fg_2}:\Ofd(\fg_2)\rightarrow \Ofd(\fg_1)$, in particular, $\widetilde\Ind_{\fg_1}^{\fg_2}$ is right exact,
			\item $\widetilde\coInd_{\fg_1}^{\fg_2}:\Ofd(\fg_1)\rightarrow \Ofd(\fg_2)$ is right adjoint to $\Res_{\fg_1}^{\fg_2}:\Ofd(\fg_2)\rightarrow \Ofd(\fg_1)$, in particular, $\widetilde\coInd_{\fg_1}^{\fg_2}$ is left exact.
		\end{itemize}
	\end{proposition}
	\begin{proof}
		Let $M \in \Ofd(\fg_1)$ and $N \in \Ofd(\fg_2)$. The usual tensor-hom adjunction gives a natural isomorphism
		\[\Hom_{\fg_2}\bigl(\Ind_{\fg_1}^{\fg_2}M,N\bigr)\simeq \Hom_{\fg_1}\bigl(M,\Res_{\fg_1}^{\fg_2}N\bigr),\]
		where on both sides we take the maps preserving the decompositions. Since $N$ is integrable over $\fg_2^{\mathrm{red}}$, every morphism $\Ind_{\fg_1}^{\fg_2}M \rightarrow N$ factors uniquely through $\widetilde\Ind_{\fg_1}^{\fg_2}M$. Hence
		\[\Hom_{\Ofd(\fg_2)}\bigl(\widetilde\Ind_{\fg_1}^{\fg_2}M,N\bigr)\simeq \Hom_{\fg_2}\bigl(\Ind_{\fg_1}^{\fg_2}M,N\bigr)\simeq\Hom_{\Ofd(\fg_1)}\bigl(M,\Res_{\fg_1}^{\fg_2}N\bigr),\]
		and all the isomorphisms are natural in $M$ and $N$. This is the required adjunction. A left adjoint functor is right exact. The second claim is proved in the same way, using that every morphism from an integrable module to $\coInd_{\fg_1}^{\fg_2}M$ factors uniquely through $\widetilde\coInd_{\fg_1}^{\fg_2}M$. A right adjoint functor is left exact.
	\end{proof}
	
	Note that the projective modules do not have to belong to the category $\Ofd(\fg_1)$. However, we claim the following.
	
	The following lemma holds in greater generality. However, the assumptions of the lemma are exactly the ones which we need, and the proof under these assumptions is much easier.
	\begin{lemma}
		Assume that $\fg_2^{\mathrm{red}}$ is isomorphic to the direct sum of a finite number of copies of $\fsl_2$. Then
		each finite-dimensional module $M \in \Ofd(\fg_1)$ admits a finite-dimensional $\widetilde\Ind_{\fg_1}^{\fg_2}$-acyclic cover:
		\[A \rightarrow M\rightarrow 0\]
		and a finite-dimensional $\widetilde\coInd_{\fg_1}^{\fg_2}$-acyclic hull
		\[0 \rightarrow M \rightarrow B.\]
	\end{lemma}
	\begin{proof}
		Let $\mathfrak{c}:=\bigoplus_{j=1}^k\fsl_2^{(j)}$ be the direct sum of the copies of $\fsl_2$ which are contained in $\fg_2$ but not in $\fg_1$. Then one has the following semidirect product 
		\[\fg_1=\sum_{j=1}^k\fn^j\rightthreetimes S,~\fn^j \subset \fsl_2^{(j)}.\]
		Consider first the following cover of a module $M$:
		\[A':=\U(\fg_1)\otimes_{\U(S)}M \simeq\Ind_S^{\fg_1}\circ \Res_S^{\fg_1}M\rightarrow M \rightarrow 0.\]
		Consider the subset 
		\[e_1^{l_1}\dots e_k^{l_k}\otimes M \subset A'.\]
		Note that the radical $R \subset \fg_2$ acts nilpotently on $M$. Therefore 
		\[B_{l_1,\dots,l_k}:=\U(\fg_1)e_1^{l_1}\dots e_k^{l_k}\otimes M \subset e_1^{l_1-\dim M}\dots e_k^{l_k-\dim M}\otimes M.\]
		Let us choose the numbers $l_1,\dots l_k$ such that for any $n \in \Ind_{\fg_1}^{(\fg_1)_{\overline 0}+(\fg_2)_{\overline 1}}M,m \in M$ we have that $\mathrm{wt}_{\mathfrak{c}} n+\mathrm{wt}_{\mathfrak{c}}(e_1^{l_1-\dim M}\dots e_k^{l_k-\dim M} \otimes m)>2$.
		Then $A\simeq A'/B_{l_1,\dots,l_k}$ is the needed cover, because 
		\[\Res_{\fb_{\mathfrak{c}}}^{(\fg_1)_{\overline 0}+(\fg_2)_{\overline 1}}\Ind_{\fg_1}^{(\fg_1)_{\overline 0}+(\fg_2)_{\overline 1}}A\]
		is isomorphic to the exterior tensor product of modules over the Borel subalgebra of $\fsl_2$ such that the weight of the cogenerator is greater than minus the weight of the generator of this module.
	\end{proof}

	This lemma allows one to define the derived functors
	\[\Lind_{\fg_1}^{\fg_2}:D^-(\Ofd(\fg_1))\rightarrow D^-(\Ofd(\fg_2)),\]
	\[\Rcoind_{\fg_1}^{\fg_2}:D^+(\Ofd(\fg_1))\rightarrow D^+(\Ofd(\fg_2)).\]

	Note that the derived functors naturally act on the Grothendieck groups. We will denote these operators by
	\[[\widetilde{\Ind}_{\fg_1}^{\fg_2}]:K_0(\Ofd(\fg_1))\rightarrow K_0(\Ofd(\fg_2)).\]
	
	We need the following property of the induction and restriction functors.
	\begin{lemma}\label{lem:commutationRestriction}
		The (derived) induction and coinduction commute with the tensor product by an integrable module. More precisely,
		let $B$ be a finite-dimensional $\fg_2$-module. Then
		\[\widetilde\Ind_{\fg_1}^{\fg_2}(\bullet \otimes \Res_{\fg_1}^{\fg_2}B)\simeq \widetilde\Ind_{\fg_1}^{\fg_2}(\bullet)\otimes B.\]
	\end{lemma}
	\begin{proof}
		For any $K \in \Ofd(\fg_2)$, $A \in \Ofd(\fg_1)$ we have the natural isomorphisms:
		\begin{multline*}
			\Hom_{\Ofd(\fg_2)}(\widetilde \Ind_{\fg_1}^{\fg_2}(A\otimes \Res_{\fg_1}^{\fg_2}B),K)\simeq \Hom_{\Ofd(\fg_1)}(A\otimes \Res_{\fg_1}^{\fg_2}B, \Res_{\fg_1}^{\fg_2}K)\simeq\Hom_{\Ofd(\fg_1)}(A, \Res_{\fg_1}^{\fg_2}K \otimes (\Res_{\fg_1}^{\fg_2}B)^*)\simeq\\
			\simeq \Hom_{\Ofd(\fg_1)}(A, \Res_{\fg_1}^{\fg_2}(K \otimes  B^*))\simeq
			\Hom_{\Ofd(\fg_2)}(\widetilde \Ind_{\fg_1}^{\fg_2}A,K \otimes  B^*)\simeq     \Hom_{\Ofd(\fg_2)}(\widetilde \Ind_{\fg_1}^{\fg_2}A\otimes  B,K ).
		\end{multline*}
		Therefore we get the needed isomorphism of functors. The same argument, starting from the adjunction of Proposition~\ref{prop:Adjoint} for the coinduction, gives
		$\widetilde\coInd_{\fg_1}^{\fg_2}(\bullet \otimes \Res_{\fg_1}^{\fg_2}B)\simeq \widetilde\coInd_{\fg_1}^{\fg_2}(\bullet)\otimes B$, and both isomorphisms extend to
		the derived functors.
	\end{proof}
	
	\begin{corollary}\label{cor:CommutationMultiplicationSymmetric}
		\[[\widetilde{\Ind}_{\fg_1}^{\fg_2}]([A \otimes \Res_{\fg_1}^{\fg_2}B])=[\widetilde{\Ind}_{\fg_1}^{\fg_2}]([A])\cdot [B].\]   
		In words, the operators on the Grothendieck groups commute with the multiplication by a restricted module.
	\end{corollary}

	We finish this subsection with a recollection of the properties of the functors related to the pair $\fb_{\fsl_2}\subset \fsl_2$.
	
	\begin{lemma}\label{lem:Sl2InductionCoinduction}
		Let $\fsl_2=\Span{e,h,f}$, let $\fb_{\fsl_2}=\Span{h,e}$ be its Borel subalgebra and let $\Bbbk_{\mu \varepsilon}$ be the one-dimensional $\fb_{\fsl_2}$-module of weight $\mu \varepsilon$, $\mu \in \mathbb{Z}$. Then
		\begin{enumerate}
			\item $\mathbf{L}^{(m)}\widetilde\Ind_{\fb_{\fsl_2}}^{\fsl_2}=0$ and $\mathbf{R}^{(m)}\widetilde\coInd_{\fb_{\fsl_2}}^{\fsl_2}=0$ for $m \geq 2$,
			\item the non-vanishing values on the one-dimensional modules are
			\[\mathbf{L}^{(0)}\widetilde\Ind(\Bbbk_{\mu \varepsilon})=V_{\mu \varepsilon}~(\mu \geq 0),\qquad \mathbf{L}^{(1)}\widetilde\Ind(\Bbbk_{\mu \varepsilon})=V_{(-\mu-2) \varepsilon}~(\mu \leq -2),\]
			\[\mathbf{R}^{(0)}\widetilde\coInd(\Bbbk_{\mu \varepsilon})=V_{-\mu \varepsilon}~(\mu \leq 0),\qquad \mathbf{R}^{(1)}\widetilde\coInd(\Bbbk_{\mu \varepsilon})=V_{(\mu-2) \varepsilon}~(\mu \geq 2),\]
			\item consequently
			\begin{equation}\label{eq:Sl2InductionCoinduction}
				\Rcoind_{\fb_{\fsl_2}}^{\fsl_2}\simeq \Lind_{\fb_{\fsl_2}}^{\fsl_2}\circ\bigl(\Bbbk_{-2\varepsilon}\otimes \bullet\bigr)[-1],
			\end{equation}
			where $\Bbbk_{-2\varepsilon}$ is the one-dimensional module of weight $-2\varepsilon$.
		\end{enumerate}
	\end{lemma}
	\begin{proof}
		Claims (1) and (2) are the Bott--Borel--Weil theorem for $\mathbb{P}^1$. On the level of the Grothendieck groups they are equivalent to
		\[[\widetilde\Ind_{\fb_{\fsl_2}}^{\fsl_2}]=(1-X^{-2})^{-1}(1-X^{-2}s),\qquad [\widetilde\coInd_{\fb_{\fsl_2}}^{\fsl_2}]=(1-X^{2})^{-1}(1-X^{2}s),\]
		the two expressions being obtained from one another by $X \mapsto X^{-1}$. Applying the first one to $X^{\mu}$ gives $\ch V_{\mu \varepsilon}$ for $\mu \geq 0$, zero for $\mu=-1$ and $-\ch V_{(-\mu-2)\varepsilon}$ for $\mu \leq -2$. Applying the second one gives $\ch V_{-\mu \varepsilon}$ for $\mu \leq 0$, zero for $\mu = 1$ and $-\ch V_{(\mu-2)\varepsilon}$ for $\mu \geq 2$.
		
		Comparing the two lines of (2) we get $\mathbf{R}^{(i)}\widetilde\coInd(\Bbbk_{\mu \varepsilon})\simeq \mathbf{L}^{(1-i)}\widetilde\Ind(\Bbbk_{(\mu-2) \varepsilon})$ for $i=0,1$. Since $\mathbf{L}^{(m)}$ sits in cohomological degree $-m$ and $\mathbf{R}^{(m)}$ in cohomological degree $m$, the passage from $1-i$ to $i$ is the shift by $[-1]$, and the twist $\mu \mapsto \mu-2$ is the tensoring by $\Bbbk_{-2\varepsilon}$. Therefore the last claim of the lemma holds on the irreducible objects. However all the indecomposable representations of the Borel subalgebra of $\fsl_2$ are tensor products of one-dimensional representations and the restrictions of integrable ones. Therefore the last claim holds because of Lemma~\ref{lem:commutationRestriction}.
	\end{proof}
	
	\subsection{Twist of a module by an automorphism}
	We need to introduce one more family of functors. Let $\fg$ be a Lie superalgebra and $\psi$ be an automorphism of $\fg$. Let $M$ be a $\fg$-module. We use the following definition.
	\begin{definition}\label{def:TwistedModule}
		$M^\psi$ is a $\fg$-module which is isomorphic to $M$ as a vector space and the action of $\fg$ is the following:
		\[X m^\psi:=(\psi(X)m)^\psi, ~\text{for }X \in \fg, m \in M.\]
	\end{definition}
	
	Then the map $\psi: M \mapsto M^{\psi}$ is an exact endofunctor of $\Ofd(\fg)$.
	
	\subsection{Functors for the parabolic subalgebras of \texorpdfstring{$\Daff$}{Daff}}
	
	In this subsection we study the induction-restriction functors in the case of parabolic subalgebras of the Lie superalgebra $\Daff$. Consider the Lie superalgebras $\fp_{ij}\subset \fp_{123}$. Consider first the restriction functor
	$\Res_{\fp_{ij}}^{\fp_{123}}$. We first compute the action of this functor on the Grothendieck groups. Recall that for $\{i,k\}=\{1,2,3\}\setminus\{j\}$
	\[K_0(\Ofd(\fp_{ik}))\simeq\mathbb{Z}[X_1^{\pm 1},X_2^{\pm 1},X_3^{\pm 1}]^{s_j}[t^{\pm 1},q^{\pm\frac{1}{2}}],~K_0(\Ofd(\fp_{123}))\simeq\mathbb{Z}[X_1^{\pm 1},X_2^{\pm 1},X_3^{\pm 1}]^{s_1,s_2,s_3}[t^{\pm 1},q^{\pm\frac{1}{2}}].\]
	
	We use these isomorphisms to identify the Grothendieck groups with the corresponding subgroups of the rings of Laurent polynomials. So the derived functors act on these rings of polynomials. In this subsection we explain these actions. 
	
	Note first that the natural inclusion
	\[[\Res_{\fp_{ik}}^{\fp_{123}}]:\mathbb{Z}[X_1^{\pm 1},X_2^{\pm 1},X_3^{\pm 1}]^{s_1,s_2,s_3}[t^{\pm 1},q^{\pm\frac{1}{2}}]\hookrightarrow \mathbb{Z}[X_1^{\pm 1},X_2^{\pm 1},X_3^{\pm 1}]^{s_j}[t^{\pm 1},q^{\pm\frac{1}{2}}]\]
	realizes the action of the restriction functor on the Grothendieck group. 
	
	Recall next that by Corollary~\ref{cor:ParabolicStable}
	the subalgebra $\fp_{ij}$ is stable under the automorphism $\pi_{(i,j)(k,0)}$, where $\{k\}=\{1,2,3\}\setminus \{i,j\}$. 
	\begin{lemma}
		The operator $[\pi_{(i,j)(k,0)}]$ acts on the ring of Laurent polynomials by the following substitution:
		\[X_k \mapsto X_k, X_i \mapsto q^{\frac{1}{2}}X_i^{-1}, X_j \mapsto q^{\frac{1}{2}}X_j^{-1}.\]
	\end{lemma}
	\begin{proof}
		Twists by automorphisms of a Lie superalgebra preserve tensor products, i.e.\ for $M,N \in \Ofd(\fp_{ij})$ we have
		\[\pi_{(i,j)(k,0)}(M) \otimes \pi_{(i,j)(k,0)}(N)\simeq \pi_{(i,j)(k,0)}(M\otimes N).\]
		Therefore $[\pi_{(i,j)(k,0)}]$ is an automorphism of the polynomial ring. Moreover, $\pi_{(i,j)(k,0)}$ is an involutive automorphism and
		\[[\pi_{(i,j)(k,0)}](X^{\alpha_i})=X^{\alpha_j},~[\pi_{(i,j)(k,0)}](X^{\alpha_k})=X^{\alpha_0}.\]
		In other words,
		\begin{equation*}
			X_kX_iX_j^{-1} \longleftrightarrow  X_kX_i^{-1}X_j,~  X_k^{-1} X_iX_j\longleftrightarrow  qX_k^{-1} X_i^{-1}X_j^{-1}.
		\end{equation*}
		This allows us to compute the action of the operator on the variables.
	\end{proof}
	
	Now let us compute the action of the induction functor 
	$[\widetilde{\Ind}_{\fp_{ij}}^{\fp_{123}}]$ on the Grothendieck group.
	\begin{lemma}\label{lem:InductionOperator}
		\[[\widetilde{\Ind}_{\fp_{ij}}^{\fp_{123}}]=(1-X_{i}^{-2})^{-1}(1-X_{i}^{-2}s_i)(1-X_{j}^{-2})^{-1}(1-X_{j}^{-2}s_j)(1-tX^{-\alpha_k})(1-tX^{-\alpha_1-\alpha_2-\alpha_3}).\]
	\end{lemma}
	\begin{proof}
		Let us denote by $F_{ij}$ the operator 
		\begin{equation}\label{eq:FDefinition}
			F_{ij}:=(1-X_{i}^{-2})^{-1}(1-X_{i}^{-2}s_i)(1-X_{j}^{-2})^{-1}(1-X_{j}^{-2}s_j)(1-tX^{-\alpha_k})(1-tX^{-\alpha_1-\alpha_2-\alpha_3}).
		\end{equation}
		Then for $f \in \mathbb{Z}[X_1^{\pm 1},X_2^{\pm 1},X_3^{\pm 1}]^{s_1,s_2,s_3}[t^{\pm 1},q^{\pm\frac{1}{2}}]$ the operator $F_{ij}$ commutes with the multiplication by $f$:
		\[F_{ij} f =f F_{ij}.\]
		Moreover by Corollary~\ref{cor:CommutationMultiplicationSymmetric} one has
		the same property for the induction operator:
		\begin{equation*}
			[\widetilde{\Ind}_{\fp_{ij}}^{\fp_{123}}]f =f[\widetilde{\Ind}_{\fp_{ij}}^{\fp_{123}}].
		\end{equation*}
		
		Now let us compute both operators on the irreducible modules $V_{ij}(\mu)$ for $\mu_i,\mu_j\geq 2$. By skew-linearity it is enough to deal with the modules without shifts. On the one hand, by direct computation we get
		\begin{multline*}
			F_{ij}\left(X_i^{\mu_i}X_j^{\mu_j} \frac{X_k^{\mu_k}-X_k^{-\mu_k-2}}{1-X_k^{-2}}\right)\\=\ch(V_{123}(\mu))-t\ch(V_{123}(\mu-\alpha_k))-t\ch(V_{123}(\mu-\alpha_1-\alpha_2-\alpha_3))+t^2\ch(V_{123}(\mu-2\varepsilon_i-2\varepsilon_j)).
		\end{multline*}
		
		On the other hand, let us compute the image of the induction functor $\Lind_{\fp_{ij}}^{\fp_{123}}$ on the irreducible module $V_{ij}(\mu)$.
		
		Consider the intermediate Lie superalgebra $\fp_{ij}^+:=\fp_{ij}\oplus \Span{e_{- \alpha_k}, e_{-\alpha_1-\alpha_2-\alpha_3}}$, $\fp_{ij}\subset \fp_{ij}^+ \subset \fp_{123}$. Then by definition we have the following decomposition of the induction functor:
		
		\begin{equation}\label{eq:CompositionInductions}
			\widetilde{\Ind}_{\fp_{ij}}^{\fp_{123}}= \widetilde{\Ind}_{\fp_{ij}^+}^{\fp_{123}}\circ \Ind^{{\fp_{ij}^+}}_{\fp_{ij}}.
		\end{equation}
		
		The functor $ \Ind^{{\fp_{ij}^+}}_{\fp_{ij}}$ is exact. Therefore \eqref{eq:CompositionInductions} can be extended to the derived functors:
		\begin{equation*}
			\Lind_{\fp_{ij}}^{\fp_{123}}= \Lind_{\fp_{ij}^+}^{\fp_{123}}\circ \Ind^{{\fp_{ij}^+}}_{\fp_{ij}}.
		\end{equation*}
		
		Consider now an irreducible $\fp_{ij}$-module $V_{ij}(\mu)$. By the definition of the induced module we have
		\begin{multline}\label{eq:FirstInductionCharacter}
			\ch  \Ind^{{\fp_{ij}^+}}_{\fp_{ij}}V_{ij}(\mu)= (1-tX^{-\alpha_k})(1-tX^{-\alpha_1-\alpha_2-\alpha_3})\ch V_{ij}(\mu)\\=(1-tX_i^{-1}X_j^{-1}X_k)(1-tX_1^{-1}X_2^{-1}X_3^{-1})\ch  V_{ij}(\mu).
		\end{multline}
		
		Assume now that $\mu_i, \mu_j \geq 2$. The two generators added to $\fp_{ij}$ have weights $-\alpha_k$ and $-\alpha_1-\alpha_2-\alpha_3$, and $\alpha_k+\alpha_1+\alpha_2+\alpha_3=2\varepsilon_i+2\varepsilon_j$. Therefore the composition series of the module $\Ind^{{\fp_{ij}^+}}_{\fp_{ij}}V_{ij}(\mu)$ consists of the four modules $V_{ij}(\mu)$, $V_{ij}(\mu-\alpha_k)\sh{0}{1}$, $V_{ij}(\mu- \alpha_1-\alpha_2-\alpha_3)\sh{0}{1}$, $V_{ij}(\mu-2\varepsilon_i-2\varepsilon_j)\sh{0}{2}$, in agreement with \eqref{eq:FirstInductionCharacter}. Since $\alpha_k$ and $\alpha_1+\alpha_2+\alpha_3$ both have coefficient $1$ at $\varepsilon_i$ and at $\varepsilon_j$, all these modules have dominant weights with respect to $h_i, h_j$. Therefore they are acyclic with respect to $\widetilde{\Ind}_{\fp_{ij}^+}^{\fp_{123}}$ by Lemma~\ref{lem:Sl2InductionCoinduction}, so that 
		\[\mathbf{L}^{(m)}\widetilde{\Ind}_{\fp_{ij}^+}^{\fp_{123}}(\Ind^{{\fp_{ij}^+}}_{\fp_{ij}}V_{ij}(\mu))=0,\qquad m>0.\]
		However
		\[\widetilde{\Ind}_{\fp_{ij}^+}^{\fp_{123}}V_{ij}(\mu)=V_{123}(\mu)\]
		and by \eqref{eq:character3Irreducible} we have
		\[\ch \widetilde{\Ind}_{\fp_{ij}^+}^{\fp_{123}}V_{ij}(\mu)=(1-X_{i}^{-2})^{-1}(1-X_{i}^{-2}s_i)(1-X_{j}^{-2})^{-1}(1-X_{j}^{-2}s_j)\ch V_{ij}(\mu).\]
		
		Together with \eqref{eq:FirstInductionCharacter} this proves Lemma~\ref{lem:InductionOperator} for the irreducible representations $V_{ij}(\mu)$, $\mu_i,\mu_j \geq 2$. In other words,
		\[([\widetilde{\Ind}_{\fp_{ij}}^{\fp_{123}}]-F_{ij})\ch V_{ij}(\mu)=0.\]
		Moreover the operator $[\widetilde{\Ind}_{\fp_{ij}}^{\fp_{123}}]-F_{ij}$ commutes with the multiplication by the character of a representation of $\fp_{123}$. In particular, it commutes with $X_i+X_i^{-1}$ and $X_j+X_j^{-1}$. Note that 
		\[\ch V_{ij}(\mu+ \varepsilon_i)=\ch V_{ij}(\mu)X_i.\]
		Assume by induction that 
		\[([\widetilde{\Ind}_{\fp_{ij}}^{\fp_{123}}]-F_{ij})\ch V_{ij}(\mu)=0, \mu_i \geq \nu_i, \mu_j \geq \nu_j.\]
		Then 
		\[([\widetilde{\Ind}_{\fp_{ij}}^{\fp_{123}}]-F_{ij})(X_i+X_i^{-1})\ch V_{ij}(\mu)=0,\]
		and hence 
		\[([\widetilde{\Ind}_{\fp_{ij}}^{\fp_{123}}]-F_{ij})\ch V_{ij}(\mu)=0, \mu_i \geq \nu_i-1, \mu_j \geq \nu_j.\]
		Then by induction we get $[\widetilde{\Ind}_{\fp_{ij}}^{\fp_{123}}]=F_{ij}$.
	\end{proof}
	
	In order for the operator $[\widetilde{\Ind}_{\fp_{ij}}^{\fp_{123}}]$ used above to be defined at all, we need the induction functor to be of finite homological dimension, cf.\ Subsection~\ref{ssec:GrothendieckGroups}. This is the content of the following lemma.
	
	\begin{lemma}\label{lem:FiniteHomologicalDimension}
		$\mathbf{L}^{(m)}\widetilde{\Ind}_{\fp_{ij}}^{\fp_{123}}=0$ for $m \geq 3$. Consequently the operator
		\[[\widetilde{\Ind}_{\fp_{ij}}^{\fp_{123}}]=\sum_{m=0}^{2}(-1)^m\bigl[\mathbf{L}^{(m)}\widetilde{\Ind}_{\fp_{ij}}^{\fp_{123}}\bigr]\]
		is well defined on the Grothendieck groups.
	\end{lemma}
	\begin{proof}
		We use the factorization \eqref{eq:CompositionInductions} through the intermediate subalgebra $\fp_{ij}^+$. The functor $\Ind^{\fp_{ij}^+}_{\fp_{ij}}$ is exact, since $\U(\fp_{ij}^+)$ is free as a right $\U(\fp_{ij})$-module. Moreover $\fp_{ij}^+$ is obtained from $\fp_{ij}$ by adding two odd root vectors only, so no integrable quotient has to be taken. Hence it suffices to bound the homological dimension of $\widetilde{\Ind}_{\fp_{ij}^+}^{\fp_{123}}$.
		
		The complement of $\fp_{ij}^+$ in $\fp_{123}$ is spanned by the two even root vectors $e_{-2\varepsilon_i}=f_i$ and $e_{-2\varepsilon_j}=f_j$. Therefore $\widetilde{\Ind}_{\fp_{ij}^+}^{\fp_{123}}$ is the composition of the two integrable induction functors corresponding to the two commuting copies of $\fsl_2$ indexed by $i$ and $j$. By part (1) of Lemma~\ref{lem:Sl2InductionCoinduction} each of them has homological dimension one, so the composition has homological dimension at most two.
	\end{proof}
	\begin{remark}
		The number $2$ obtained here is exactly the shift appearing in Theorem~\ref{thm:IsomorphismInductionCoinduction}.
	\end{remark}
	
	\begin{definition}\label{def:Fdefinition}
		Let us introduce the functor 
		\[\mathcal{F}_{ij}:=\Lind_{\fp_{ij}}^{\fp_{123}} \circ \pi_{(i,j)(k,0)}\circ {\Res}_{\fp_{ij}}^{\fp_{123}}.\]
	\end{definition}
	
	The computations of this subsection immediately give us
	\begin{proposition}\label{prop:FEqual}
		\[[\mathcal F_{ij}]=F_{ij} \circ [\pi_{(i,j)(k,0)}].\]
	\end{proposition}

	\section{Categorification of the genus two DAHA}\label{sec:Categorification}
	
	The genus two DAHA $\mathcal{A}_{q,t}$ was introduced in~\cite{ArthamonovShakirov-2019} as the algebra generated by three multiplication operators $\hat{O}_{B_{ij}}=X_{ij}+X_{ij}^{-1}$ together with three commuting difference operators
	
	\begin{equation}\label{eq:OAOriginalForm}
		\hat O_{A_{ij}}=\hat O_{A_k}=\sum_{a,b\in\{\pm1\}}C^{a,b}_{i,j}\delta_{i}^a \delta_{j}^b,
	\end{equation}
	where
	\begin{equation}\label{eq:Cab}
		C^{a,b}_{i,j}=ab\,
		\frac{\bigl(1-t\,X_{k}X_{i}^{a}X_{j}^{b}\bigr)\bigl(1-t\,X_{k}^{-1}X_{i}^{a}X_{j}^{b}\bigr)}
		{t\,X_{i}^{a}X_{j}^{b}\bigl(X_{i}-X_{i}^{-1}\bigr)\bigl(X_{j}-X_{j}^{-1}\bigr)},
	\end{equation}
	and $\delta_{i}^a$ are the $q$-shift operators given by the substitutions
	\[\delta_{i}^a(X_i)=q^{\frac{a}{2}} X_i.\]
	
	The operators $\hat{O}_{A_i}$ are the genus two analogues of the Macdonald operators, their basis of common eigenfunctions being the genus two Macdonald polynomials $\Psi_{j_1,j_2,j_3}$. The mapping class group of $\Sigma_2$ acts by automorphisms on this algebra~\cite{ArthamonovShakirov-2019}, while the specialization $t=q$ is isomorphic to the skein algebra $Sk_q(\Sigma_2)$~\cite{CookeSamuelson-2021}. In this section we construct a categorification of $\mathcal{A}_{q,t}$. Namely, the six generators lift to functors, the parameters $q,t$ arise from gradings and $\mathcal{A}_{q,t}$ is recovered on the Grothendieck group.
	
	Recall the reflections from the Weyl group which act by $s_{k}\colon X_{k}\mapsto X_{k}^{-1}$, fixing the two remaining $X$'s.
	
	\begin{lemma}
		Let $f$ be a symmetric function, $f =s_i(f)=s_j(f)$. Then 
		\[\delta_{i}^a \delta_{j}^bf=s_i^{\frac{a+1}{2}}s_j^{\frac{b+1}{2}}[\pi_{(i,j)(k,0)}]f.\]
	\end{lemma}
	\begin{proof}
		We have
		\[[\pi_{(i,j)(k,0)}] s_i f = [\pi_{(i,j)(k,0)}] f.\]
		Thus 
		\[\delta_{i}^a\delta_{j}^b f=s_i^{\frac{a+1}{2}}s_j^{\frac{b+1}{2}}[\pi_{(i,j)(k,0)}]s_i^{\frac{a-1}{2}}s_j^{\frac{b-1}{2}}f=s_i^{\frac{a+1}{2}}s_j^{\frac{b+1}{2}}[\pi_{(i,j)(k,0)}]f.\]
	\end{proof}
	
	Therefore, acting on symmetric functions, the operator \eqref{eq:OAOriginalForm} takes the form
	\begin{equation}\label{eq:OAOriginal}
		\hat O_{A_{ij}}=\hat O_{A_k}=\sum_{a,b\in\{\pm1\}}C^{a,b}_{i,j}\,s_{i}^{\frac{a+1}{2}}s_{j}^{\frac{b+1}{2}}\,[\pi_{(i,j)(k,0)}].
	\end{equation}

	\begin{proposition}\label{prop:OAF}
		Recall the operator $F_{ij}$ of~\eqref{eq:FDefinition}. One has $\hat O_{A_{ij}}=t^{-1}F_{ij}\,[\pi_{(i,j)(k,0)}]$.
	\end{proposition}
	
	\begin{proof}
		Let us multiply both sides of the desired equality from the right by the involution $[\pi_{(i,j)(k,0)}]$. We get the following equality equivalent to the initial one:
		\begin{equation}\label{eq:core}
			\sum_{a,b\in\{\pm1\}}C^{a,b}_{i,j}\,
			s_{i}^{\frac{a+1}{2}}s_{j}^{\frac{b+1}{2}}=t^{-1}F_{ij}.
		\end{equation}
		
		Since $X_{i}-X_{i}^{-1}=X_{i}\bigl(1-X_{i}^{-2}\bigr)$, and similarly for $j$, the factor $\bigl(1-X_{i}^{-2}\bigr)^{-1} \bigl(1-X_{j}^{-2}\bigr)^{-1}$ is independent of $a,b$ and may be pulled out of the sum:
		\begin{equation}\label{eq:step1}
			\text{LHS of }\eqref{eq:core}=\frac{1}{t}\bigl(1-X_{i}^{-2}\bigr)^{-1}\bigl(1-X_{j}^{-2}\bigr)^{-1}\!\!\sum_{a,b\in\{\pm1\}}\!\! ab\,\frac{\bigl(1-t X_{k}X_{i}^{a}X_{j}^{b}\bigr)\bigl(1-t X_{k}^{-1}X_{i}^{a}X_{j}^{b}\bigr)} {X_{i}^{a+1}X_{j}^{b+1}}\,s_{i}^{\frac{a+1}{2}}s_{j}^{\frac{b+1}{2}}.
		\end{equation}
		Define now the group elements $\sigma_{a,b}:=s_{i}^{\frac{a+1}{2}}s_{j}^{\frac{b+1}{2}}$. If $a=1$ then $\sigma_{a,b}$ contains $s_{i}$ and sends $X_{i}^{a}=X_{i}$ to $X_{i}^{-1}$. If $a=-1$ then $\sigma_{a,b}$ fixes $X_{i}$ and the exponent is already $-1$. Either way $\bigl(X_{i}^{a}\bigr)^{\sigma_{a,b}}=X_{i}^{-1}$, and similarly for $X_{j}$. Hence for every $a,b\in\{\pm1\}$,
		\begin{align}
			\bigl(1-t X_{k}X_{i}^{a}X_{j}^{b}\bigr)
			\bigl(1-t X_{k}^{-1}X_{i}^{a}X_{j}^{b}\bigr)\sigma_{a,b}
			&=\sigma_{a,b}\bigl(1-t X_{k}X_{i}^{-1}X_{j}^{-1}\bigr)
			\bigl(1-t X_{k}^{-1}X_{i}^{-1}X_{j}^{-1}\bigr)
			\notag\\
			&=\sigma_{a,b}\bigl(1-t X^{-\alpha_k}\bigr)
			\bigl(1-t X^{-\alpha_1-\alpha_2-\alpha_3}\bigr),
			\label{eq:step2}
		\end{align}
		the last equality holds because
		$X_{k}X_{j}^{-1}X_{i}^{-1}=X^{-\alpha_k}$ and
		$X_{k}^{-1}X_{j}^{-1}X_{i}^{-1}=X^{-\alpha_i-\alpha_j-\alpha_k}$.
		The right-hand factors no longer depend on $a,b$, so they too leave the sum.
		
		Since $s_{i}$ fixes $X_{j}$ and $s_{j}$ fixes $X_{i}$, the sum
		factorizes, and for a single index
		\begin{equation}
			\sum_{a\in\{\pm1\}}a\,X_{i}^{-a-1}s_{i}^{\frac{a+1}{2}}
			=X_{i}^{-2}s_{i}-1=-\bigl(1-X_{i}^{-2}s_{i}\bigr).
		\end{equation}
		When the two factors are multiplied, the signs cancel:
		\begin{equation}\label{eq:step3}
			\sum_{a,b\in\{\pm1\}}ab\,X_{i}^{-a-1}X_{j}^{-b-1}\,
			s_{i}^{\frac{a+1}{2}}s_{j}^{\frac{b+1}{2}}
			=\bigl(1-X_{i}^{-2}s_{i}\bigr)\bigl(1-X_{j}^{-2}s_{j}\bigr).
		\end{equation}
		Substituting \eqref{eq:step2} and \eqref{eq:step3} into \eqref{eq:step1}
		yields \eqref{eq:core}, which completes the proof.
	\end{proof}
	
	We can now state our main categorification result.
	\begin{theoremB}\label{thm:MainTheorem}
		We have the following action of the functors $\mathcal{F}_{ij}$ on the Grothendieck group:
		\[[\mathcal{F}_{ij}]=t \hat O_{A_{ij}}.\]
		In other words, the functors $\mathcal{F}_{ij}$ categorify the genus two Macdonald operators.
	\end{theoremB}
	\begin{proof}
		By Proposition~\ref{prop:FEqual} one has $[\mathcal{F}_{ij}]=F_{ij}\circ[\pi_{(i,j)(k,0)}]$, and by Proposition~\ref{prop:OAF} $\hat O_{A_{ij}}=t^{-1}F_{ij}[\pi_{(i,j)(k,0)}]$, whence $[\mathcal{F}_{ij}]=t\hat O_{A_{ij}}$. 
	\end{proof}
	
	\section{The \texorpdfstring{$\Ext$}{Ext} pairing and orthogonality}\label{sec:Orthogonality}
	In this section we introduce the graded $\Ext$ pairing, compute it explicitly and prove that the operators $\hat O_{A_{ij}}$ are self-adjoint with respect to it.
	
	\subsection{Graded \texorpdfstring{$\Hom$}{Hom} and \texorpdfstring{$\Ext$}{Ext} pairing}\label{ssec:ExtPairing}
	In this subsection we introduce the $\Ext$ pairing between $\mathcal{O}^{\gr}(\fp_{123})$ and $\Ofd(\fp_{123})$ and its restriction to the Grothendieck group of $\Ofd(\fp_{123})$. It is computed explicitly in Subsection~\ref{ssec:ProjectiveModules}.
	
	For the pairing we need infinite-dimensional graded modules over the parabolic subalgebras $\fp_J$.
	
	\begin{definition}\label{def:GradedCategory}
		We denote by $\mathcal{O}^{\gr}(\fp_J)$ the category of $\fp_J$-modules $M$, integrable over $\fp_J^{\mathrm{red}}$, with a decomposition $M=\bigoplus_{\lambda,k,r}M_{\lambda,k,r}$ as in Definition~\ref{def:CategoryO} such that the $z$-degrees and the odd degrees of $M$ are bounded from below and for every $k$ the space $\bigoplus_{\lambda,r}M_{\lambda,k,r}$ is finite-dimensional. Morphisms are the $\fp_J$-linear maps preserving the decomposition.
	\end{definition}
	The shifts $M\sh{l}{r}$ of Definition~\ref{def:ShiftDefinition} are defined in the same way for $M\in\mathcal{O}^{\gr}(\fp_J)$.
	
	For $M,N\in\Ofd(\fp_J)$ we write $\Ext^i(M,N):=\Hom_{D^b(\Ofd(\fp_J))}(M,N[i])$, and for $M\in\mathcal{O}^{\gr}(\fp_{123})$, $N\in\Ofd(\fp_{123})$ we write $\Ext^i(M,N):=\Hom_{D^b(\mathcal{O}^{\gr}(\fp_{123}))}(M,N[i])$. For $M\in\Ofd(\fp_{123})$ the two definitions agree, see Lemma~\ref{lem:ExtComparison}. Recall the grading shifts $M\sh{l}{r}$ introduced in Definition~\ref{def:ShiftDefinition}. We have the natural isomorphisms
	\[\Hom\bigl(M ,N\sh{l}{r}\bigr)\simeq \Hom\bigl(M\sh{-l}{-r} ,N\bigr).\]
	
	We consider the graded sums of these vector spaces
	\[\Hom_{\gr}(M ,N):=\bigoplus_{l\in\frac12\mathbb{Z},\, r \in \mathbb{Z}}\Hom\bigl(M ,N\sh{l}{r}\bigr),\qquad
	\Ext^{i}_{\gr}(M ,N):=\bigoplus_{l\in\frac12\mathbb{Z},\, r \in \mathbb{Z}}\Ext^{i}\bigl(M ,N\sh{l}{r}\bigr),\]
	regarded as $\frac12\mathbb{Z}\times\mathbb{Z}$-graded vector spaces, the summand indexed by $(l,r)$ sitting in degree $(l,r)$. For complexes $M^\bullet,N^\bullet\in D^b(\Ofd(\fp_J))$ we define $\Hom_{\gr}(M^\bullet,N^\bullet)$ in the same way, with $\Hom_{D^b(\Ofd(\fp_J))}$ in place of $\Hom$.
	
	For a $\frac12\mathbb{Z}\times\mathbb{Z}$-graded vector space $V=\bigoplus_{l,r}V_{l,r}$ we write
	\begin{equation*}
		\dim_{q,t}(V):=\sum_{l\in\frac12\mathbb{Z},\, r \in \mathbb{Z}}\dim(V_{l,r})\,q^l(-t)^r,
	\end{equation*}
	with the same normalization of $t$ as in \eqref{eq:CharacterDefinition}.
	
	In the rest of this subsection $J=\{1,2,3\}$ and $R$ is the radical of $\fp_{123}$. The sum of the $z$-degree and the odd degree is called the total degree. Every homogeneous element of $R$ has non-negative $z$-degree, non-negative odd degree and positive total degree, and $\fp_{123}^{\mathrm{red}}$ lies in total degree $0$. The part of $R$ of $z$-degree $0$ is odd, and the bracket of two odd elements vanishes.
	
	\begin{lemma}\label{lem:Projectivity}
		For any irreducible $\fp_{123}^{\mathrm{red}}$-module $V$ the module
		\[P(V):=\U(\fp_{123})\otimes_{\U(\fp_{123}^{\mathrm{red}})}V\]
		is projective in $\mathcal{O}^{\gr}(\fp_{123})$.
	\end{lemma}
	\begin{proof}
		As a graded $\fp_{123}^{\mathrm{red}}$-module $P(V)\simeq \U(R)\otimes V$. The part of $R$ of $z$-degree $0$ is odd with zero bracket, and the other components of $R$ are finite-dimensional of positive $z$-degree. Hence the components of $\U(R)$ of fixed $z$-degree are finite-dimensional, and $P(V) \in \mathcal{O}^{\gr}(\fp_{123})$.
		By the tensor-hom adjunction, for any $M \in \mathcal{O}^{\gr}(\fp_{123})$ we have a natural isomorphism
		\[\Hom_{\fp_{123}}(P(V),M)\simeq \Hom_{\fp_{123}^{\mathrm{red}}}\bigl(V,\Res^{\fp_{123}}_{\fp_{123}^{\mathrm{red}}}M\bigr).\]
		The algebra $\fp_{123}^{\mathrm{red}}$ is reductive and $M$ is integrable over it, hence $\Res^{\fp_{123}}_{\fp_{123}^{\mathrm{red}}}M$ is a semisimple $\fp_{123}^{\mathrm{red}}$-module and the functor $\Hom_{\fp_{123}^{\mathrm{red}}}(V,\bullet)$ is exact on the image of the restriction. Therefore $\Hom_{\fp_{123}}(P(V),\bullet)$ is exact, i.e.\ $P(V)$ is projective.
		
		Since $R$ raises the total degree and $V$ is the component of $P(V)$ of least total degree, $P(V)$ has a unique maximal submodule, namely $R\cdot P(V)$, and $P(V)/R\cdot P(V)\simeq V$. Hence $P(V)$ is the projective cover of the irreducible module with the same highest weight.
	\end{proof}
	
	The same argument shows that $\U(\fp_{123})\otimes_{\U(\fp_{123}^{\mathrm{red}})}W$ is projective in $\mathcal{O}^{\gr}(\fp_{123})$ for every $\fp_{123}^{\mathrm{red}}$-module $W$, integrable and graded as in Definition~\ref{def:CategoryO}, for which it lies in $\mathcal{O}^{\gr}(\fp_{123})$.

	\begin{lemma}\label{lem:ExtComparison}
		For $M,N\in\Ofd(\fp_{123})$ the inclusion $\Ofd(\fp_{123})\subset\mathcal{O}^{\gr}(\fp_{123})$ induces isomorphisms between $\Ext^i(M,N)$ in $\Ofd(\fp_{123})$ and in $\mathcal{O}^{\gr}(\fp_{123})$.
	\end{lemma}
	
	\begin{definition}
		For $M\in\mathcal{O}^{\gr}(\fp_{123})$ and $N\in\Ofd(\fp_{123})$ put
		\[(M,N)_{\Ext}:=\sum_{i=0}^{\infty}(-1)^i\dim_{q,t}(\Ext^{i}_{\gr}(M ,N)) \in \mathbb{Z}((t))((q^{\frac{1}{2}})).\]
	\end{definition}
	
	The sum is well defined. Indeed, for fixed $(l,r)$ the spaces $\Ext^i(M,N\sh{l}{r})$ are finite-dimensional and vanish for large $i$.
	
	Consider short exact sequences
	\[0 \rightarrow M_1 \rightarrow M \rightarrow M_2 \rightarrow 0,\]
	\[0 \rightarrow N_1 \rightarrow N \rightarrow N_2 \rightarrow 0.\]
	
	For fixed $(l,r)$ the $\Ext$ long exact sequences are finite sequences of finite-dimensional spaces, and they give the following equalities:
	\begin{equation*}
		(M,N)_{\Ext}=(M_1,N)_{\Ext}+(M_2,N)_{\Ext}=(M,N_1)_{\Ext}+(M,N_2)_{\Ext}.
	\end{equation*}
	See for example~\cite{weibel1994introduction}.
	Therefore the value of the $\Ext$ pairing depends only on the classes of the modules in the Grothendieck groups.
	
	Moreover, by definition the $\Ext$ pairing is $q,t$-skew-linear:
	\begin{equation*}
		\bigl(M\sh{l}{r},N\bigr)_{\Ext}=(M,N)_{\Ext}\,q^l(-t)^r=\bigl(M,N\sh{-l}{-r}\bigr)_{\Ext}.
	\end{equation*}
	
	Let $(-,-): K_0( \mathcal{O}^{\gr}(\fp_{123}))\times K_0( \Ofd(\fp_{123})) \rightarrow \mathbb{Z}((t))((q^{\frac{1}{2}}))$ be the pairing on the Grothendieck groups induced by the $\Ext$ pairing:
	\begin{equation*}
		([M],[N]):=(M,N)_{\Ext}.
	\end{equation*}
	
	Then this is a skew $\mathbb{Z}[q^{\pm \frac{1}{2}},t^{\pm 1}]$-bilinear form, i.e.\
	\[([M]q^lt^r,[N])=([M],[N])q^lt^r=([M],[N]q^{-l}t^{-r}).\]
	
	The inclusion $\Ofd(\fp_{123})\subset\mathcal{O}^{\gr}(\fp_{123})$ is exact and induces a map $K_0(\Ofd(\fp_{123}))\rightarrow K_0(\mathcal{O}^{\gr}(\fp_{123}))$. Composing with it, we obtain a pairing on $K_0(\Ofd(\fp_{123}))$, denoted by the same symbol. By Lemma~\ref{lem:ExtComparison} it is computed by $\Ext$ in $\Ofd(\fp_{123})$. The characters \eqref{eq:CharacterDefinition} of the objects of $\mathcal{O}^{\gr}(\fp_{123})$ lie in $\mathbb{Z}[X_1^{\pm 1},X_2^{\pm 1},X_3^{\pm 1}]^{s_1,s_2,s_3}[t^{\pm 1}]((q^{\frac{1}{2}}))$, and this defines the map $\iota_{123}$ also on $K_0(\mathcal{O}^{\gr}(\fp_{123}))$.
	
	For an irreducible module $V_{123}(\mu)$ let $P_{123}(\mu)$ be its projective cover, as in Lemma~\ref{lem:Projectivity}. Projective modules are $\Ext$-acyclic, i.e.\
	\[\Ext^{i}(P_{123}(\mu) ,N)=\{0\},~ i \geq 1.\]
	Moreover
	\[\dim \Hom\bigl(P_{123}(\mu)\sh{l}{r},V_{123}(\mu')\sh{l'}{r'} \bigr)=\delta_{\mu,\mu'}\delta_{l,l'}\delta_{r,r'}.\]
	
	Therefore one gets
	\begin{equation*}
		\bigl(P_{123}(\mu)\sh{l}{r},V_{123}(\mu')\sh{l'}{r'}\bigr)_{\Ext} =q^{l-l'}(-t)^{r-r'}\delta_{\mu,\mu'}.
	\end{equation*}
	
	The classes of the projective modules are dual to the classes of the irreducible modules, i.e.\
	\begin{equation*}
		([P_{123}(\mu)],[V_{123}(\mu')])=\delta_{\mu,\mu'}.
	\end{equation*}
	
	This pairing extends to the derived category. For $M^\bullet, N^\bullet \in D^b(\Ofd(\fp_{123}))$ one can define
	
	\begin{definition}
		\[(M^\bullet,N^\bullet )_{\Ext}:=\sum_{i=-\infty}^{\infty}(-1)^i\dim_{q,t}\bigl(\Hom_{\gr}(M^\bullet ,N^\bullet[i])\bigr),\]
		where $[i]$ is the homological shift.
	\end{definition}
	
	If the modules $M,N$ are identified with the complexes concentrated in degree zero, then both definitions of the $\Ext$ pairing agree by Lemma~\ref{lem:ExtComparison}. By induction on the number of non-zero cohomology objects, the truncation triangles and the long exact $\Hom$ sequences show that the sum is well defined and that $(M^\bullet,N^\bullet)_{\Ext}=([M^\bullet],[N^\bullet])$.
	
	\subsection{Projective modules}\label{ssec:ProjectiveModules}
	We now compute the characters of the projective modules. This will give the pairing on the Grothendieck group in an explicit form.
	
	Let $R$ be the pro-nilpotent radical of the Lie superalgebra $\fp_{123}$, so that $\fp_{123}\simeq \fp_{123}^{\mathrm{red}}\ltimes R$ with $\fp_{123}^{\mathrm{red}}\simeq \fsl_2\oplus \fsl_2 \oplus \fsl_2$ spanned by $\{e_i,h_i,f_i\}$ in $z$-degree zero. Then
	\[R=\Span{e_{\pm 2 \varepsilon_i}z^k,\ h_iz^k\ (i=1,2,3,\ k \geq 1),\ e_{\pm\varepsilon_1\pm\varepsilon_2\pm\varepsilon_3}z^k\ (k \geq 0)}.\]
	The three families $h_iz^k$, $k \geq 1$, are the root vectors of the imaginary roots $k\delta$, whose multiplicity is three, and they produce the factor $(q;q)_\infty^{-3}$.
	
	Note that
	\[P_{123}(\mu) \simeq \U(\fp_{123})\otimes_{\U(\fp_{123}^{\mathrm{red}})}V_{123}(\mu).\]
	
	Therefore there is an isomorphism of graded vector spaces
	\[P_{123}(\mu) \simeq \U(R)\otimes V_{123}(\mu).\]
	
	Therefore we have the following equality of characters:
	\[\ch (P_{123}(\mu))=\ch( \U(R))\ch ( V_{123}(\mu)).\]
	
	Note that
	\begin{equation*}
		\ch( \U(R))=\frac{\prod_{l \geq 0}\prod_{a,b,c\in\{\pm 1\}}(1-X_1^aX_2^bX_3^c q^lt)}{\prod_{l \geq 1}\prod_{i=1}^3(1-X_i^2 q^l)(1-X_i^{-2} q^l)}(q;q)_\infty^{-3},
	\end{equation*}
	where $(q;q)_\infty:=\prod_{l \geq 1}(1-q^l)$ is the $q$-Pochhammer symbol.
	
	Thus we have
	\begin{equation*}
		\left(\frac{\prod_{l \geq 0}\prod_{a,b,c\in\{\pm 1\}}(1-X_1^aX_2^bX_3^c q^lt)}{\prod_{l \geq 1}\prod_{i=1}^3(1-X_i^2 q^l)(1-X_i^{-2} q^l)}[V_{123}(\mu)],[V_{123}(\mu')]\right)=\delta_{\mu,\mu'}(q;q)_\infty^{3}.
	\end{equation*}
	
	For $f \in \mathbb{Z}[X_1^{\pm 1},X_2^{\pm 1},X_3^{\pm 1}]((t))((q^{\frac{1}{2}}))$ we denote by $f_0 \in \mathbb{Z}((t))((q^{\frac{1}{2}}))$ its $X$-free term. Then for modules over a finite-dimensional semisimple Lie algebra $\fg$ we have
	\[\left(\prod_{\alpha \in \Phi_{\fg}}(1-X^\alpha)\ch(V_\lambda)\ch(V_{\lambda'})\right)_0=|W|\,\delta_{\lambda^*,\lambda'}.\]
	Here $\Phi_\fg$ is the root system and $W$ the Weyl group of the Lie algebra $\fg$, and $V_{\lambda^*}:=V_\lambda^*$. Note that $\fp_{123}^{\mathrm{red}}\simeq \fsl_2 \oplus \fsl_2 \oplus \fsl_2$, so that $|W|=8$ and $V_{123}(\mu)^*\simeq V_{123}(\mu)$.
	
	Thus
	\[\left(\prod_{i=1}^3(1-X_i^2)(1-X_i^{-2})\ch V_{123}(\mu)\ch V_{123}(\mu')\right)_0=8\,\delta_{\mu,\mu'}.\]
	
	This equality can be rephrased in the following way:
	\begin{equation}\label{eq:BilinearFormOnBasis}
		\frac18\,(q;q)_\infty^3\left(\frac{\prod_{l \geq 0}\prod_{i=1}^3(1-X_i^2 q^l)(1-X_i^{-2} q^l)}{\prod_{l \geq 0}\prod_{a,b,c\in\{\pm 1\}}(1-X_1^aX_2^bX_3^c q^lt)}[P_{123}(\mu)][V_{123}(\mu')]\right)_0=\delta_{\mu,\mu'}.
	\end{equation}
	
	Let us define
	\[\Delta:=\frac18\,(q;q)_\infty^3\frac{\prod_{l \geq 0}\prod_{i=1}^3(1-X_i^2 q^l)(1-X_i^{-2} q^l)}{\prod_{l \geq 0}\prod_{a,b,c\in\{\pm 1\}}(1-X_1^aX_2^bX_3^c q^lt)},\]
	expanded as a series in $q^{\frac12}$ and $t$.
	
	For $f \in \mathbb{Z}[X_1^{\pm 1},X_2^{\pm 1},X_3^{\pm 1}][t^{\pm 1},q^{\pm\frac{1}{2}}]$ we denote by $f^\star$ the substitution
	\[f^\star:=f|_{t \mapsto t^{-1},\,q \mapsto q^{-1}}.\]
	
	Then, extending the equality \eqref{eq:BilinearFormOnBasis} by skew-linearity, one gets the following lemma.
	
	\begin{lemma}\label{lem:ExtPairingExplicit}
		Let $f \in K_0(\mathcal{O}^{\gr}(\fp_{123}))$, $g \in K_0(\Ofd(\fp_{123}))$.
		The $\Ext$ pairing on the Grothendieck ring can be expressed in the following form:
		\[(f,g)=(f \Delta g^\star)_0.\]
	\end{lemma}
	
	\subsection{Orthogonality of operators}
	Throughout this subsection $\{i,j,k\}=\{1,2,3\}$, we abbreviate $\pi:=\pi_{(i,j)(k,0)}$, and $M,N$ are objects of $\Ofd(\fp_{123})$.
	
	By Proposition~\ref{prop:Adjoint} the functor ${\Lind}_{\fp_{ij}}^{\fp_{123}}$ is left adjoint to $\Res_{\fp_{ij}}^{\fp_{123}}$, so that for every $A\in\Ofd(\fp_{ij})$ and $n \in \mathbb{Z}$
	\[\Hom_{\gr}(A[n],\Res_{\fp_{ij}}^{\fp_{123}}N)\simeq
	\Hom_{\gr}({\Lind}_{\fp_{ij}}^{\fp_{123}}A[n],N).\]
	Recall also that $\pi$ is an involution, hence for the twist functor of Definition~\ref{def:TwistedModule} one has $\Hom(A^{\pi},B)\simeq \Hom(A,B^{\pi})$: a $\fp_{ij}$-morphism $A^\pi \rightarrow B$ is the same thing as a $\fp_{ij}$-morphism $A \rightarrow B^{\pi^{-1}}=B^{\pi}$. Combining the adjunction for $\Lind$, this twist and the adjunction for $\Rcoind$ we get
	\begin{multline*}
		\Hom_{\gr}({\Lind}_{\fp_{ij}}^{\fp_{123}}\circ \pi \circ \Res_{\fp_{ij}}^{\fp_{123}}M[n],N)\simeq  \Hom_{\gr}(\pi \circ \Res_{\fp_{ij}}^{\fp_{123}} M[n],\Res_{\fp_{ij}}^{\fp_{123}} N)\\ \simeq 
		\Hom_{\gr}(\Res_{\fp_{ij}}^{\fp_{123}}M[n],\pi \circ \Res_{\fp_{ij}}^{\fp_{123}} N)\simeq  \Hom_{\gr}(M[n],{\Rcoind}_{\fp_{ij}}^{\fp_{123}}\circ \pi \circ \Res_{\fp_{ij}}^{\fp_{123}} N).
	\end{multline*}
	
	Therefore the following operators on the Grothendieck group are adjoint:
	\begin{equation}\label{eq:AdjointOperators}
		(t \hat O_{A_{ij}}[M],[N])=([{\Lind}_{\fp_{ij}}^{\fp_{123}}\circ \pi \circ \Res_{\fp_{ij}}^{\fp_{123}}][M],[N])= 
		([M],[{\Rcoind}_{\fp_{ij}}^{\fp_{123}}\circ \pi \circ \Res_{\fp_{ij}}^{\fp_{123}}][N]).
	\end{equation}
	
	It remains to identify the operator on the right-hand side. It is possible to compute it directly, but we prove a somewhat stronger claim on the level of functors.
	\begin{theoremA}\label{thm:IsomorphismInductionCoinduction}
		There is the following isomorphism of derived functors:
		\[{\Rcoind}_{\fp_{ij}}^{\fp_{123}}\circ \pi \circ \Res_{\fp_{ij}}^{\fp_{123}}\simeq {\Lind}_{\fp_{ij}}^{\fp_{123}}\circ \pi \circ  \Res_{\fp_{ij}}^{\fp_{123}}\sh{0}{-2}[-2].\]
	\end{theoremA}
	
	\begin{proof}
		It is enough to prove that 
		\begin{equation}\label{eq:IndCoindComparison}
			{\Rcoind}_{\fp_{ij}}^{\fp_{123}}\simeq {\Lind}_{\fp_{ij}}^{\fp_{123}}\sh{0}{-2}[-2].
		\end{equation}
		
		\emph{Step 1: factorization.} Recall the intermediate superalgebra $\fp_{ij}^+$ from the proof of Lemma~\ref{lem:InductionOperator}. The functors ${\Ind}_{\fp_{ij}}^{\fp_{ij}^+}$ and ${\coInd}_{\fp_{ij}}^{\fp_{ij}^+}$ are exact, because the quotient $\fn:=\fp_{ij}^+/\fp_{ij}$ is purely odd of dimension $(0|2)$ and no integrable quotient or submodule has to be taken. Therefore
		\begin{equation}\label{eq:FactorizationOfBoth}
			{\Lind}_{\fp_{ij}}^{\fp_{123}}\simeq {\Lind}_{\fp_{ij}^+}^{\fp_{123}}\circ {\Ind}_{\fp_{ij}}^{\fp_{ij}^+},\qquad {\Rcoind}_{\fp_{ij}}^{\fp_{123}} \simeq {\Rcoind}_{\fp_{ij}^+}^{\fp_{123}}\circ {\coInd}_{\fp_{ij}}^{\fp_{ij}^+}.
		\end{equation}
		
		\emph{Step 2: the odd part.} Note that $\fn$ is a two-dimensional $\fp_{ij}$-module (it is the quotient $\fp_{ij}^+/\fp_{ij}$, which is a module. The span of $e_{-\alpha_k}$ and $e_{-\alpha_1-\alpha_2-\alpha_3}$ inside $\fp_{ij}^+$ is not a $\fp_{ij}$-submodule, for instance $[e_{2\varepsilon_i},e_{-\alpha_k}]$ is proportional to $e_{\alpha_j}$). Its top exterior power is one-dimensional of weight $-\alpha_k-\alpha_1-\alpha_2-\alpha_3=-2\varepsilon_i-2\varepsilon_j$ and of odd degree $2$, i.e.\ $\Lambda^{\mathrm{top}}\fn \simeq V_{ij}(-2\varepsilon_i-2\varepsilon_j)\sh{0}{2}$. By the Poincar\'e--Birkhoff--Witt theorem $\U(\fp_{ij}^+)\simeq \U(\fp_{ij})\otimes \Lambda(\fn)$, whence
		\begin{equation} \label{eq:OddInduction}
			{\Ind}_{\fp_{ij}}^{\fp_{ij}^+}(\bullet)\simeq {\coInd}_{\fp_{ij}}^{\fp_{ij}^+}\bigl(\bullet \otimes V_{ij}(-2\varepsilon_i-2\varepsilon_j)\sh{0}{2}\bigr).
		\end{equation}
		Indeed, \eqref{eq:OddInduction} follows from the isomorphisms of $\fh$-modules ${\Ind}_{\fp_{ij}}^{\fp_{ij}^+}M \simeq \Lambda(\fn)\otimes M$ and ${\coInd}_{\fp_{ij}}^{\fp_{ij}^+}M \simeq \Lambda(\fn)^*\otimes M$, together with $\Lambda(\fn)\simeq \Lambda(\fn)^*\otimes \Lambda^{\mathrm{top}}\fn$. Equivalently,
		\begin{equation}\label{eq:OddCoinduction}
			{\coInd}_{\fp_{ij}}^{\fp_{ij}^+}(\bullet)\simeq {\Ind}_{\fp_{ij}}^{\fp_{ij}^+}\bigl(\bullet \otimes V_{ij}(2\varepsilon_i+2\varepsilon_j)\sh{0}{-2}\bigr).
		\end{equation}
		
		\emph{Step 3: the even part.} The complement of $\fp_{ij}^+$ in $\fp_{123}$ is spanned by $f_i$ and $f_j$, so that $\widetilde\Ind_{\fp_{ij}^+}^{\fp_{123}}$ and $\widetilde\coInd_{\fp_{ij}^+}^{\fp_{123}}$ are the functors attached to the two commuting copies of $\fsl_2$ indexed by $i$ and $j$. Applying \eqref{eq:Sl2InductionCoinduction} twice we get
		\begin{equation}\label{eq:EvenIndCoind}
			{\Rcoind}_{\fp_{ij}^+}^{\fp_{123}} \simeq {\Lind}_{\fp_{ij}^+}^{\fp_{123}}\circ \bigl(\Bbbk_{-2\varepsilon_i-2\varepsilon_j}\otimes \bullet \bigr)[-2].
		\end{equation}
		
		\emph{Step 4: the two twists cancel.} The modules $V_{ij}(\pm(2\varepsilon_i+2\varepsilon_j))$ are one-dimensional, so tensoring by them commutes with induction. Substituting \eqref{eq:OddCoinduction} and \eqref{eq:EvenIndCoind} into the second identity of \eqref{eq:FactorizationOfBoth} we obtain
		\begin{multline*}
			{\Rcoind}_{\fp_{ij}}^{\fp_{123}}(\bullet)\simeq {\Lind}_{\fp_{ij}^+}^{\fp_{123}}\Bigl({\Ind}_{\fp_{ij}}^{\fp_{ij}^+}\bigl(\bullet \otimes V_{ij}(2\varepsilon_i+2\varepsilon_j)\sh{0}{-2}\bigr)\otimes \Bbbk_{-2\varepsilon_i-2\varepsilon_j}\Bigr)[-2]\\
			\simeq {\Lind}_{\fp_{ij}^+}^{\fp_{123}}{\Ind}_{\fp_{ij}}^{\fp_{ij}^+}\bigl(\bullet \otimes V_{ij}(2\varepsilon_i+2\varepsilon_j)\otimes V_{ij}(-2\varepsilon_i-2\varepsilon_j)\bigr)\sh{0}{-2}[-2]\simeq {\Lind}_{\fp_{ij}}^{\fp_{123}}(\bullet)\sh{0}{-2}[-2],
		\end{multline*}
		which is \eqref{eq:IndCoindComparison}.
	\end{proof}
	
	\begin{corollary}\label{cor:CoindOperator}
		\[[{\Rcoind}_{\fp_{ij}}^{\fp_{123}}\circ \pi \circ \Res_{\fp_{ij}}^{\fp_{123}}]=t^{-1} \hat O_{A_{ij}}.\]
	\end{corollary}
	\begin{proof}
		By Theorem~\ref{thm:IsomorphismInductionCoinduction} the left-hand side equals $(-t)^{-2}(-1)^{-2}[\mathcal{F}_{ij}]$, the first factor coming from the shift $\sh{0}{-2}$ by Definition~\ref{def:ShiftDefinition} and the second one from the homological shift $[-2]$. They equal $t^{-2}$ and $1$, respectively. Since $[\mathcal{F}_{ij}]=t\hat O_{A_{ij}}$ by Theorem~\ref{thm:MainTheorem}, the claim follows.
	\end{proof}
	
	\begin{theoremD}\label{thm:SelfAdjointness}
		The operators $\hat O_{A_{ij}}$ are self-adjoint with respect to the pairing $(-,-)$ on $K_0(\Ofd(\fp_{123}))$:
		\begin{equation*}
			(\hat O_{A_{ij}}f,g)=(f,\hat O_{A_{ij}}g).
		\end{equation*}
	\end{theoremD}
	\begin{proof}
		Combining \eqref{eq:AdjointOperators} with Corollary~\ref{cor:CoindOperator} we get $(t\hat O_{A_{ij}}f,g)=(f,t^{-1}\hat O_{A_{ij}}g)$. The pairing takes values in $\mathbb{Z}((t))((q^{\frac{1}{2}}))$, is $\mathbb{Z}[q^{\pm\frac{1}{2}},t^{\pm 1}]$-linear in the first argument and skew-linear in the second one, so the left-hand side equals $t(\hat O_{A_{ij}}f,g)$ and the right-hand side equals $t(f,\hat O_{A_{ij}}g)$. Cancelling $t$ gives the claim.
	\end{proof}
	
	\begin{corollary}\label{cor:Orthogonality}
		Let $\Psi_{i_1,i_2,i_3}$ and $\Psi_{j_1,j_2,j_3}$ be genus two Macdonald polynomials. If $(i_1,i_2,i_3) \neq (j_1,j_2,j_3)$, then they are orthogonal, i.e.
		\[(\Psi_{i_1,i_2,i_3},\Psi_{j_1,j_2,j_3})=0.\]
	\end{corollary}
	\begin{proof}
		By Theorem~\ref{thm:SelfAdjointness}, $\lambda_{ij}(f,g)=(\hat O_{A_{ij}}f,g)=(f,\hat O_{A_{ij}}g)=\mu_{ij}^{\star}(f,g)$, the last equality by the skew-linearity of the pairing in the second argument.
	\end{proof}
	
	\section{Proof of Theorem~\ref{thm:CommFunctors}}\label{sec:ProofC}
	The goal of this section is to prove the commutation relation for the functors $\mathcal{F}_{ij}$. For a symbol $j=i_1i_2\dots i_r$, where $i_k =0,1,2,3$, we write
	$\alpha_j:=\alpha_{i_1}+\alpha_{i_2}+\dots+\alpha_{i_r}$, and for a finite set $J$ of
	such symbols put
	\begin{equation}\label{eq:IwahoriNotation}
		\J{J}\ :=\ \Iw\ \oplus\ \bigoplus_{j\in J}\Bbbk\,\ee{-\alpha_j}.
	\end{equation}
	The letter $\Iw$ is never written: every subalgebra below is spelled by the list of the
	negative root vectors adjoined to the Iwahori subalgebra. 
	In particular, the parabolic subalgebras in the new notation are denoted as follows:
	\[
	\fp_{ij}=\J{i,j,ij},\quad \fp_{ijk}=\J{i,j,k,ij,ik,jk,ijk}.
	\]
	
	\medskip\noindent
	\textbf{Notation for the functors.} We write the restriction,
	induction, coinduction and Zuckerman functors with a single subscript: the list of the words whose
	root vectors the functor deletes or adjoins. Thus $\Ind_{3}$ adjoins $\ee{-\alpha_3}$, and
	$\Res_{0,01,012}$ deletes $\ee{-\alpha_0},\ee{-\alpha_0-\alpha_1}$ and
	$\ee{-\alpha_0-\alpha_1-\alpha_2}$. The two algebras between which the functor acts are
	not written. Only the rightmost functor of a composition carries the algebra it
	starts from, after a vertical bar, and every other source and target is then forced. Thus
	\[\Ind_{3}\,\Zuck_{01}\,\Res_{3}\Big|_{\J{0,1,2,3,012}}\]
	is read from the right: $\Res_{3}$ goes from $\J{0,1,2,3,012}$ to $\J{0,1,2,012}$,
	$\widehat{\Zuck}_{01}:=\Res_{01}\Rcoind_{01}$ and $\Zuck_{01}:=\Res_{01}\Lind_{01}$ act on $\J{0,1,2,012}$, and $\Ind_{3}$ returns to $\J{0,1,2,3,012}$.
	
	For $\gamma\in P\oplus\frac12\mathbb{Z}\delta$ we also write $X^{\gamma}$ for the functor of tensor product with the one-dimensional module of weight $\gamma$ and odd degree $0$, whenever this module exists, for instance on $\Ofd(\J{J})$ if all the roots $\alpha_j$, $j\in J$, are odd. On characters this functor acts by multiplication by $X^{\gamma}$.

	\subsection{Commutation relations for \texorpdfstring{$\Ind$}{Ind}, \texorpdfstring{$\Res$}{Res} and \texorpdfstring{$\Zuck$}{Z}}\label{sec:Lem}
	
	\begin{lemma}\label{lem:trans}
		For a chain of Lie superalgebras $\fg_1\subset \fg_2 \subset \fg_3$
		\[\Res^{\fg_3}_{\fg_1}=\Res^{\fg_2}_{\fg_1}\circ\Res^{\fg_3}_{\fg_2},\qquad
		\Lind^{\fg_3}_{\fg_1}\simeq\Lind^{\fg_3}_{\fg_2}\circ\Lind^{\fg_2}_{\fg_1}.\]
	\end{lemma}
	
	\begin{lemma}[base change]\label{lem:bc}
		Let $\fg_1, \fg_2$ be two Lie superalgebras such that their sum $\fg_1+\fg_2$ is also a Lie superalgebra. Then
		\[\Res^{{\fg_1+\fg_2}}_{{\fg_2}}\circ\Lind^{{\fg_1+\fg_2}}_{\fg_1}\ \simeq\
		\Lind^{\fg_2}_{\fg_2 \cap \fg_1}\circ\Res^{\fg_1}_{\fg_1 \cap \fg_2}\]
		if 
		\begin{enumerate}	\item[(i)] either $(\fg_2)_{\overline 0}\subset \fg_1$, i.e.\ all the inductions are odd,
			\item[(ii)] or $(\fg_2)_{\overline 1}\subset \fg_1$, i.e.\ the inductions are even, $\fg_2=\fg_1\cap \fg_2 \oplus \fn_-$ as vector spaces, where $\fn_-$ is a negative nilpotent subalgebra of a semisimple Lie subalgebra $\fg\subset (\fg_2)_{\overline 0}$.
		\end{enumerate}
	\end{lemma}
	Note that this list is not exhaustive, i.e.\ the above functors may commute even when neither of the two sufficient conditions of the lemma holds.

	\subsection{The functor \texorpdfstring{$\mathcal F_{12}\circ\mathcal F_{13}$}{F12 F13} and its decomposition}\label{sec:Dec}

	We need to prove $\mathcal{F}_{12} \circ \mathcal{F}_{13} \simeq \mathcal{F}_{13} \circ \mathcal{F}_{12}$. From now on and until the end of the section we omit the composition symbols $\circ$. 
	
	Note that
	\[\mathcal{F}_{13}   \mathcal{F}_{12}\simeq  \pi_{(2,3)}  \mathcal{F}_{12}   \mathcal{F}_{13}   \pi_{(2,3)}^{-1}.\]
	We will prove that the composition $\mathcal{F}_{12}   \mathcal{F}_{13}$ is stable under the $\pi_{(2,3)}$ conjugation, i.e.\ under the interchange of indices $2 \leftrightarrow 3$.
	
	Recall that $\mathcal F_{ij}=\Lind^{\fp_{123}}_{\fp_{ij}} \pi_{(i,j)(k,0)} 
	\Res^{\fp_{123}}_{\fp_{ij}}$. The functor whose stability we study is
	\begin{equation}\label{eq:N1}\,\mathcal F_{12}\,\mathcal F_{13}\ =\
		\Lind_{3,13,23,123}\,\pi_{(1,2)(3,0)}\,\Res_{3,13,23,123}\,\Lind_{2,12,23,123}\,\pi_{(1,3)(2,0)}\,\Res_{2,12,23,123}\Big|_{\J{1,2,3,12,13,23,123}} .
	\end{equation}
	
	Since $\pi_{(1,3)(2,0)}$ preserves \J{1,3,13} and carries \J{1,2,3,12,13,23,123} to \J{0,1,3,01,03,13,013}, it moves through the
	last restriction:
	\begin{equation}\label{eq:N2}
		=\ \Lind_{3,13,23,123}\,\pi_{(1,2)(3,0)}\,\Res_{3,13,23,123}\,\Lind_{2,12,23,123}\,\Res_{0,01,03,013}\Big|_{\J{0,1,3,01,03,13,013}}\,\pi_{(1,3)(2,0)} .
	\end{equation}
	Moving $\pi_{(1,2)(3,0)}$ to the right in the same way, by means of
	\begin{gather*}\,\pi_{(1,2)(3,0)}\Bigl(\J{1,2,12}\Bigr)=\J{1,2,12},\qquad
		\pi_{(1,2)(3,0)}\Bigl(\J{1,3,13}\Bigr)=\J{0,2,02},\\
		\pi_{(1,2)(3,0)}\Bigl(\J{1,2,3,12,13,23,123}\Bigr)=\J{0,1,2,01,02,12,012},\\
		\pi_{(1,2)(3,0)}\Bigl(\J{0,1,3,01,03,13,013}\Bigr)=\J{0,2,3,02,03,23,023},
	\end{gather*}
	and using $\pi_{(1,2)(3,0)}\,\pi_{(1,3)(2,0)}=\pi_{(0,1)(2,3)}$, we get
	\begin{equation}\label{eq:N3}\,\mathcal F_{12}\,\mathcal F_{13}\ =\ \mathcal G\,\pi_{(0,1)(2,3)},\qquad
		\pi_{(0,1)(2,3)}\colon\ \Ofd\Bigl(\J{1,2,3,12,13,23,123}\Bigr)\longrightarrow
		\Ofd\Bigl(\J{0,2,3,02,03,23,023}\Bigr),
	\end{equation}
	where
	\begin{equation}\label{eq:L1}\,\mathcal G\ :=\ \Lind_{3,13,23,123}\,\Res_{0,01,02,012}\,\Lind_{1,01,12,012}\,\Res_{3,03,23,023}\Big|_{\J{0,2,3,02,03,23,023}} .
	\end{equation}
	We have the following equality of actions of the symmetric group:
	\[\pi_{(2,3)}\,\pi_{(0,1)(2,3)}\,\pi_{(2,3)}^{-1}\simeq \pi_{(0,1)(2,3)}.\]
	Thus we are to prove that $\mathcal G$ is stable under the interchange
	$\alpha_2\leftrightarrow\alpha_3$.
	
	\subsection*{Decomposition of inductions}
	
	By Lemma~\ref{lem:trans} the third factor of $\mathcal G$ splits as
	\begin{equation}\label{eq:L2}\,\Lind_{1,01,12,012}\ \simeq\ \Lind_{12}\,\Lind_{1,01,012}\Big|_{\J{0,2,02}}.
	\end{equation}
	By the second claim of Lemma~\ref{lem:bc}:
	\begin{equation}\label{eq:L3}\,\Res_{0,01,02,012}\,\Lind_{12}\ \simeq\ \Lind_{12}\,\Res_{0,01,02,012}\Big|_{\J{0,1,2,01,02,012}} .
	\end{equation}
	
	By Lemma~\ref{lem:trans} again,
	\begin{equation}\label{eq:L6}\,\Lind_{3,13,23,123}\,\Lind_{12}\ \simeq\ \Lind_{3,12,13,23,123}\ \simeq\ \Lind_{12,13}\,\Lind_{3,23,123}\Big|_{\J{1,2}} ,
	\end{equation}
	and the first factor of the right-hand side commutes with $\pi_{(2,3)}$. Substituting \eqref{eq:L2}, \eqref{eq:L3},
	\eqref{eq:L6} into \eqref{eq:L1} and dropping that factor, we are left with
	\begin{equation}\label{eq:L8}\,\mathcal{H}:=\Lind_{3,23,123}\,\Res_{0,01,02,012}\,\Lind_{1,01,012}\,\Res_{3,03,23,023}\Big|_{\J{0,2,3,02,03,23,023}} .
	\end{equation}
	So we are to prove that $\pi_{(2,3)}\,\mathcal{H}\,\pi_{(2,3)}^{-1} \simeq \mathcal{H}$.

	Consider the splitting of the third factor of \eqref{eq:L8} into its even and odd parts,
	\begin{equation}\label{eq:R1}\,\Lind_{1,01,012}\ \simeq\ \Lind_{01}\,\Ind_{1,012}\Big|_{\J{0,2,02}} ,
	\end{equation}
	and the second factor as
	\begin{equation}\label{eq:B2}\,\Res_{0,01,02,012}\ =\ \Res_{0,01,012}\,\Res_{02}\Big|_{\J{0,1,2,01,02,012}}.
	\end{equation}
	
	Thus we have
	\[ \mathcal{H}\simeq\Lind_{3,23,123}\,\Res_{0,01,012}\,\Res_{02}\,\Lind_{01}\,\Ind_{1,012}\,\Res_{3,03,23,023}\Big|_{\J{0,2,3,02,03,23,023}} \]
	and hence we may move the restriction with respect to the root $\alpha_0+\alpha_2$ to the right by applying Lemma~\ref{lem:bc} twice:
	\[\mathcal{H} \simeq \Lind_{3,23,123}\,\Res_{0,01,012}\,\Lind_{01}\,\Ind_{1,012}\,\Res_{3,02,03,23,023}\Big|_{\J{0,2,3,02,03,23,023}}.\]
	
	The last factor splits, by Lemma~\ref{lem:trans}, as
	\begin{equation}\label{eq:B2b}\,\Res_{3,02,03,23,023}\ =\ \Res_{3,23,023}\,\Res_{02,03}\Big|_{\J{0,2,3,02,03,23,023}} ,
	\end{equation}
	and the right factor commutes with $\pi_{(2,3)}$.
	
	Hence we need to prove the stability of the following functor under the interchange $2 \leftrightarrow 3$:
	\begin{equation}\label{eq:B3}\,\mathcal{E} := \Lind_{3,23,123}\,\Res_{0,01,012}\
		\ \Lind_{01}\,\Ind_{1,012}\,\Res_{3,23,023}\Big|_{\J{0,2,3,23,023}} .
	\end{equation}
	
	To study this functor we first split it into the composition of minimal inductions and restrictions:

	\begin{equation}\label{eq:M1}\,\mathcal{E} \simeq\ \Lind_{23}\,\Ind_{123}\,\rd{\Ind_{3}}\,\Res_{0}\,\Res_{012}\,\Res_{01}\,\Lind_{01}\,\Ind_{012}\,\Ind_{1}\,\rd{\Res_{3}}\,\Res_{023}\,\Res_{23}\Big|_{\J{0,2,3,23,023}} .
	\end{equation}
	Using Lemma~\ref{lem:bc} one gets the following isomorphism of functors:
	\begin{equation}\label{eq:M2}\,\mathcal{E} \simeq\
		\underbrace{\Lind_{23}\,\Ind_{123}\,\Res_{0}\,\Res_{012}}_{\vdefeq{A}}\
		\underbrace{\rd{\Ind_{3}}\,\Res_{01}\,\Lind_{01}\,\rd{\Res_{3}}}_{\vdefeq{B}}\
		\underbrace{\Ind_{012}\,\Ind_{1}\,\Res_{023}\,\Res_{23}}_{\vdefeq{C}}\Big|_{\J{0,2,3,23,023}} .
	\end{equation}
	
	\subsection{Study of the functor \texorpdfstring{$B$}{B}}
	
	In this subsection we interpret the functor $B$ as the cone of a natural transformation of two functors. Note that the cone of a natural transformation gives a map on objects and morphisms. However, in general it is not a functor. We prove that in our case this cone is in fact isomorphic to the functor. Later we use the $\pi_{(2,3)}$-stability of the natural transformation to obtain the corresponding stability of the cone, which is identified with a functor.
	
	First of all we need the following relations between induction and coinduction.
	\begin{itemize}
		\item Let $\fg_2=\adj{\fg_1}{\ee{-\alpha}}$ with $\alpha$ odd. Then
		\begin{equation}\label{eq:IndCoIndOdd}\,\Ind^{\fg_2}_{\fg_1}\ \simeq\ \coInd^{\fg_2}_{\fg_1}\,X^{-\alpha}\sh{0}{1} .
		\end{equation}
		\item
		Recall Lemma~\ref{lem:Sl2InductionCoinduction}. We will use it for the $\fsl_2$ of the even root $\alpha_0+\alpha_1$,
		\begin{equation}\label{eq:IndCoIndEven}\,\Rcoind_{01}\ \simeq\
			\Lind_{01}\,X^{-\alpha_0-\alpha_1}[-1]\Big|_{\J{0,1,2,3,012,013}},
		\end{equation}
		and the same for the pair $\J{0,1,2,012}\subset\J{0,1,2,01,012}$.
		
		Let us construct three natural transformations between endofunctors of the category $\Ofd\bigl(\J{0,1,2,3,012}\bigr)$.
	\end{itemize}
	
	\subsubsection{Construction of the natural transformation $\varphi$.}
	
	To shorten the notation, denote the algebra $\fa:=\J{0,1,2,3,012}$ and the functor
	\[
	B':=\Res_{013}\,\Zuck_{01}\,\Ind_{013}\Big|_{\fa}.
	\]
	We compute $\Hom(B,B')$ by applying adjunctions one
	after another.
	
	\begin{equation}\label{eq:HomChain}
		\newcommand{\viso}{\rotatebox[origin=c]{90}{$\sim$}}
		\begin{tikzcd}[row sep=large, column sep=large]
			\Hom\Bigl(\Ind_{3}\,\Zuck_{01}\,\Res_{3}\Big|_{\fa}\ ,\ \
			\Res_{013}\,\Zuck_{01}\,\Ind_{013}\Big|_{\fa}\Bigr)
			\arrow[d, "\viso"]
			\arrow[r, phantom, "\ni"]
			& \varphi \\
			\Hom\Bigl(\Zuck_{01}\,\Res_{3}\Big|_{\fa}\ ,\ \
			\Res_{3}\,\Res_{013}\,\Zuck_{01}\,\Ind_{013}\Big|_{\fa}\Bigr)
			\arrow[d, "\viso"]
			& \\
			\Hom\Bigl(\Res_{3}\Big|_{\fa}\ ,\ \
			\widehat{\Zuck}_{01}\,\Res_{3}\,\Res_{013}\,\Zuck_{01}\,\Ind_{013}\Big|_{\fa}\Bigr)
			\arrow[d, "\viso"]
			& \\
			\Hom\Bigl(\Res_{3}\Big|_{\fa}\ ,\ \
			\Res_{3}\,\Res_{013}\,\Zuck_{01}\,\Ind_{013}\Big|_{\fa}\Bigr)
			\arrow[r, phantom, "\ni"]
			& \Res_{3}\Bigl(\mathrm{Id}\Big|_{\fa}\longrightarrow
			\Res_{013}\,\Zuck_{01}\,\Ind_{013}\Big|_{\fa}\Bigr)
			\arrow[uuu, mapsto]
		\end{tikzcd}
	\end{equation}
	Here we used the relation
	\[\widehat{\Zuck}_{01}\,\Res_{3}\,\Res_{013}\,\Zuck_{01}\,\Ind_{013}\Big|_{\fa} \simeq \Res_{3}\,\Res_{013}\,\widehat{\Zuck}_{01}\,\Zuck_{01}\,\Ind_{013}\Big|_{\fa}\simeq \Res_{3}\,\Res_{013}\,\Zuck_{01}\,\Ind_{013}\Big|_{\fa},\]
	where the first isomorphism follows from Lemma~\ref{lem:bc}.
	
	\smallskip
	\subsubsection{Construction of the natural transformation $j[1]$.} Split $B'$ into its four elementary steps and replace both
	inductions by coinductions, using \eqref{eq:IndCoIndEven} for the even one and
	\eqref{eq:IndCoIndOdd} for the odd one:
	\begin{multline}
		B'=\ \Res_{013}\,\Res_{01}\,\Lind_{01}\,\Ind_{013}
		\simeq\ \Res_{013}\,\Res_{01}\,\Rcoind_{01}\  
		\qquad \ X^{\alpha_0+\alpha_1}[1]  \coInd_{013}\,X^{-\alpha_0-\alpha_1-\alpha_3}\sh{0}{1}\\
		\simeq\ \Res_{013}\,\Res_{01}\,\Rcoind_{01}\,\coInd_{013}\ 
		\qquad \ X^{-\alpha_3}\sh{0}{1}\,[1]\Big|_{\J{0,1,2,3,012}} .\label{eq:BpCoInd}
	\end{multline}
	
	Multiplying the counit
	$\Res_{01,013}\,\Rcoind_{01,013}\to\mathrm{Id}$
	from the right by $X^{-\alpha_3}\sh{0}{1}[1]$ gives
	\begin{equation}\label{eq:mapj}
		j[1]\colon\quad B'\ \simeq\ \Res_{01,013}\,\Rcoind_{01,013}\,X^{-\alpha_3}\sh{0}{1}\,[1]\Big|_{\J{0,1,2,3,012}}\ \longrightarrow\ X^{-\alpha_3}\sh{0}{1}\,[1].
	\end{equation}
	
	\subsubsection{Object-wise exact triangle}
	Write
	$\fb_{\alpha_0+\alpha_1}:=\Span{h_{\alpha_0+\alpha_1},\ee{\alpha_0+\alpha_1}}$ for the Borel
	subalgebra of the $\fsl_2$ governing $\Lind_{01}$.
	Let
	\[
	\mathrm{pr}\colon\fa\longrightarrow\fb_{\alpha_0+\alpha_1}
	\]
	be the projection and $\mathrm{pr}^*W$ the pullback of a $\fb_{\alpha_0+\alpha_1}$-module $W$. Denote
	by $V_{n,n'}$ the $\fb_{\alpha_0+\alpha_1}$-module of lowest weight $n\varepsilon$ and highest weight
	$n'\varepsilon$, the weights being those of Lemma~\ref{lem:Sl2InductionCoinduction} for this $\fsl_2$.
	
	\begin{lemma}\label{lem:String}
		$V_{n,n'}\simeq\Res V_{\frac{n'-n}{2}\varepsilon}\otimes\Bbbk_{\frac{n+n'}{2}\varepsilon}$, and
		\[
		\mathbf R^{(m)}\widetilde\coInd_{01}V_{n,n'}=0\ \ (m>0)
		\qquad\Longleftarrow\qquad
		n+n'\leq2 .
		\]
	\end{lemma}
	
	\begin{proof}
		Both modules have weights $n\varepsilon,(n+2)\varepsilon,\dots,n'\varepsilon$ with
		$\ee{\alpha_0+\alpha_1}$ acting by a regular nilpotent, hence are isomorphic. Since
		$V_{\frac{n'-n}{2}\varepsilon}$ is finite-dimensional, coinduction commutes with tensoring by it, so
		\[
		\mathbf R^{(m)}\widetilde\coInd_{01}V_{n,n'}\simeq
		V_{\frac{n'-n}{2}\varepsilon}\otimes\mathbf R^{(m)}\widetilde\coInd_{01}\bigl(\Bbbk_{\frac{n+n'}{2}\varepsilon}\bigr),
		\]
		and by Lemma~\ref{lem:Sl2InductionCoinduction} the right-hand side vanishes for $m>0$  when
		$\frac{n+n'}{2}\leq1$.
	\end{proof}
	
	\begin{lemma}\label{lem:Good}
		Let $M$ be a finite-dimensional $\fa$-module such that
		\begin{equation}\label{eq:RightTail}
			\Res_{\fb_{\alpha_0+\alpha_1}}M\ \simeq\ \bigoplus_{k\in K}V_{n_k,n'_k},\qquad n_k+n'_k\leq-2\quad\text{for all }k .
		\end{equation}
		Then $H^{m}B(M)=H^{m}B'(M)=0$ for $m\neq -1$.
	\end{lemma}
	
	\begin{proof}
		Such an $M$ has a filtration whose subquotients are the modules
		$\mathrm{pr}^*V_{n_k,n'_k}$, so by the long exact sequence it suffices to treat
		$N:=\mathrm{pr}^*V_{n,n'}$ with $n+n'\leq -2$.
		
		The functors $\Res_3,\Res_{013},\Ind_3,\Ind_{013}$ are exact. The $\fb_{\alpha_0+\alpha_1}$-modules
		$X^{\alpha_0+\alpha_1}\Res_3N$, $X^{\alpha_0+\alpha_1}\Ind_{013}N$ have filtrations by modules $V_{n,n'}$ with $n+n'\leq 2$, so these modules are acyclic with respect to $\Rcoind_{01}$ by
		Lemma~\ref{lem:String}. Therefore the lemma holds because of Serre duality between $\Rcoind_{01}$ and $\Lind_{01}$.
	\end{proof}
	
	\begin{lemma}\label{lem:SES}
		Let $M$ be an $\fa$-module satisfying condition \eqref{eq:RightTail}. Then there exists the following short exact sequence:
		\[0 \rightarrow 
		B M \xrightarrow{\ \varphi\ }
		B' M \xrightarrow{\ j[1]\ }X^{-\alpha_3}\sh{0}{1} M \rightarrow 0   .\]
	\end{lemma}
	\begin{proof}
		First we rewrite the claim in terms of coinduction functors. By Lemma~\ref{lem:Good}, all the non-trivial homology is concentrated in degree $-1$, and the claim is equivalent to
		\begin{equation}
			0 \rightarrow 
			\Ind_{3}\ \widehat{\Zuck}_{01}\ X^{\alpha_0+\alpha_1}\Res_{3}\ M \xrightarrow{\ \varphi\ }
			\Res_{013}\ \widehat{\Zuck}_{01}\ \Ind_{013}\ X^{\alpha_0+\alpha_1}M \xrightarrow{\ j[1]\ }X^{-\alpha_3}\sh{0}{1} M \rightarrow 0.
		\end{equation}
		Note first that it is enough to prove this lemma for $M \simeq \mathrm{pr}^* V_{n,n'}$. Indeed, assume that we have a short exact sequence 
		\[0 \rightarrow M_1 \rightarrow M \rightarrow M_2 \rightarrow 0\]
		of modules satisfying the condition \eqref{eq:RightTail}. Assume that the claim holds for $M_1,M_2$. Then using Lemma~\ref{lem:Good} and the $3 \times 3$ lemma we obtain the claim for $M$. Therefore the sequence is exact for any module which has a composition series with subquotients isomorphic to $\mathrm{pr}^*V_{n_k,n'_k}$. Let us prove the lemma for modules $\mathrm{pr}^*V_{n_k,n'_k}$. Consider first a one-dimensional module. Then the short exact sequence is the following:
		\begin{center}
			\begin{tikzpicture}[>={Stealth[round]},font=\small,
				vtx/.style={circle,fill,inner sep=1.2pt},
				gvtx/.style={circle,fill=black!35,inner sep=1.2pt},
				rcirc/.style={draw=red,thick,circle,inner sep=0.5pt,minimum size=0.32cm}]
				
				\def\s{0.95}      %% horizontal step of a row
				\def\d{0.475}     %% half step
				\def\h{0.80}      %% height of a tooth
				\def\xo{-2.375}   %% x of the leftmost vertex of the lower rows
				\def\xl{8.60}     %% centre line of the right-hand column
				\def\ya{-1.90}    %% spine of the comb
				\def\yb{-4.10}    %% lower row of the strip
				\def\yc{-6.30}    %% lower row of the last diagram
				
				%% ===================== right-hand column: the sequence =====================
				\node (l0) at (\xl,-0.05) {$0$};
				\node (la) at (\xl,\ya+0.40) {$\Ind_{3}\,\Res_{01}\,\widetilde{\coInd}_{01}\,X^{\alpha_0+\alpha_1}\,\Res_{3}\ 
					\qquad $};
				\node (lb) at (\xl,\yb+0.40) {$\Res_{013}\,\Res_{01}\,\widetilde{\coInd}_{01}\,\coInd_{013}\ 
					X^{-\alpha_3}\sh{0}{1}$};
				\node (lc) at (\xl,\yc+0.40) {$X^{-\alpha_3}\sh{0}{1}$};
				\node (l1) at (\xl,-7.38) {$0$};
				\draw[->] (l0) -- (la);
				\draw[->] (la) -- (lb);
				\draw[->] (lb) -- (lc);
				\draw[->] (lc) -- (l1);
				
				%% ===================== 0 =====================
				\node at (0,-0.05) {$0$};
				\draw[->] (0,-0.28) -- (0,-0.72);
				
				%% ===================== first diagram: the comb =====================
				\foreach \k in {0,...,5}{\node[vtx] (s\k) at (\xo+\k*\s,\ya) {};}
				\foreach \k in {0,...,4}{\pgfmathtruncatemacro{\l}{\k+1}\draw[->] (s\k) -- (s\l);}
				\foreach \k in {0,...,5}{\node[vtx] (u\k) at (\xo+\k*\s+\d,\ya+\h) {};}
				\foreach \k in {0,...,5}{\draw[->] (s\k) -- (u\k);}
				\node[anchor=north west] at (\xo+0.06,\ya-0.10) {$\ee{\alpha_0+\alpha_1}$};
				\node[anchor=west] at (\xo+5*\s+0.38,\ya+0.34) {$\ee{-\alpha_3}$};
				
				\draw[->] (0,\ya-0.35) -- (0,\yb+\h+0.45);
				
				%% ===================== second diagram: the strip =====================
				\foreach \k in {0,...,5}{\node[vtx] (p\k) at (\xo+\k*\s,\yb) {};}
				\foreach \k in {0,...,4}{\pgfmathtruncatemacro{\l}{\k+1}\draw[->] (p\k) -- (p\l);}
				\foreach \k in {0,...,6}{\node[vtx] (q\k) at (\xo+\k*\s-\d,\yb+\h) {};}
				\foreach \k in {0,...,5}{\pgfmathtruncatemacro{\l}{\k+1}\draw[->] (q\k) -- (q\l);}
				\foreach \k in {0,...,5}{\pgfmathtruncatemacro{\l}{\k+1}\draw[->] (p\k) -- (q\k);\draw[->] (p\k) -- (q\l);}
				\node[rcirc] (c2) at (q0) {};
				\node[anchor=north west] at (\xo+0.06,\yb-0.10) {$\ee{\alpha_0+\alpha_1}$};
				\node[anchor=west] at (\xo+5*\s+0.38,\yb+0.34) {$\ee{-\alpha_3}$};
				
				\draw[->] (0,\yb-0.35) -- (0,\yc+\h+0.40);
				
				%% ===================== third diagram: one vertex =====================
				%% the lattice of the strip, in grey
				\foreach \k in {0,...,5}{\node[gvtx] at (\xo+\k*\s,\yc) {};}
				\foreach \k in {1,...,6}{\node[gvtx] at (\xo+\k*\s-\d,\yc+\h) {};}
				%% the only vertex of the module
				\node[vtx] (x0) at (\xo-\d,\yc+\h) {};
				\node[rcirc] (c3) at (x0) {};
				
				%% ===================== 0 =====================
				\draw[->] (0,\yc-0.35) -- (0,-7.15);
				\node at (0,-7.38) {$0$};
				
			\end{tikzpicture}
			\captionof{figure}{The short exact sequence of Lemma~\ref{lem:SES} for a one-dimensional module $M$.}
		\end{center}
		The general case follows from the one-dimensional case because all these functors commute with tensoring by the restriction of an integrable module.
	\end{proof}
	
	We denote $\fc:=[0,1,2,3]$. Consider modules $M \in \Ofd(\fc)$ satisfying the following conditions:
	\begin{equation}\label{eq:COnditions}
		\Res^{\fc}_{\fb_{\alpha_0+\alpha_1}}M\simeq \bigoplus_{k\in K}V_{n_k,n'_k},\qquad n_k+n'_k\leq-2\quad\text{for all }k .
	\end{equation}
	
	These conditions imply that $\Ind_{012}M$ satisfies \eqref{eq:RightTail}. Therefore we have the following.
	
	\begin{lemma}
		Let $M$ be a $\fc$-module satisfying condition \eqref{eq:COnditions}. Then there exists the following short exact sequence:
		\[0 \rightarrow 
		\Res_{012}B \Ind_{012}M \xrightarrow{\ \varphi\ }
		\Res_{012}B' \Ind_{012}M \xrightarrow{\ j[1]\ }X^{-\alpha_3}\sh{0}{1} \Res_{012}\Ind_{012}M \rightarrow 0   .\]     
	\end{lemma}
	
	Using the unit and the counit of the odd induction and restriction, we get the following short exact sequence of exact functors on the abelian category:
	
	\begin{equation}
		0 \rightarrow X^{-\alpha_3}\sh{0}{1}\rightarrow X^{-\alpha_3}\sh{0}{1} \Res_{012}\Ind_{012} \xrightarrow{\eta}  \sh{-1}{2} \rightarrow 0  
	\end{equation}
	
	Therefore we have the following exact sequence for any $M$ satisfying \eqref{eq:COnditions}:
	\begin{equation}
		0 \rightarrow 
		\Res_{012}B \Ind_{012}M \rightarrow  \ker(\eta \circ \varphi(M))\rightarrow X^{-\alpha_3}\sh{0}{1}M \rightarrow 0.
	\end{equation}
	Note that each complex of objects of $\Ofd(\fc)$ is quasi-isomorphic to a complex of modules satisfying \eqref{eq:COnditions}. Therefore we have the following exact triangle:
	\begin{equation}
		\Res_{012}B \Ind_{012} \rightarrow  \mathbf{R}\ker(\eta \circ \varphi)[1]\rightarrow X^{-\alpha_3}\sh{0}{1}[1] \rightarrow +1.
	\end{equation}
	
	\begin{lemma}
		\[\Lind_{23}\,X^{-\alpha_3}\sh{0}{1}\,\Ind_{123}\, 
		\Res_{0}\,\Ind_{1}\,
		\Res_{023}\,\Res_{23}\Big|_{\J{0,2,3,23,023}} \simeq 0.\]   
	\end{lemma}
	\begin{proof}	
		This functor is isomorphic to
		\[
		\Ind_{1,123}\  \
		\Lind_{23}\,X^{-\alpha_3}\sh{0}{1}\,  
		\Res_{23}\, 
		\Res_{0,023}\Big|_{\J{0,2,3,23,023}} \]
		and $\Lind_{23}\,X^{-\alpha_3}\sh{0}{1}\,
		\Res_{23}\Big|_{\J{2,3,23}} \simeq 0$
		because $-\alpha_3(h_{\alpha_2+\alpha_3})=-1$.
	\end{proof}
	
	\begin{corollary}
		\[A\circ B\circ C \simeq \Lind_{23}\Ind_{123}\,
		\Res_{0}\mathbf{R}\ker(\eta \circ \varphi)[1]\Ind_{1}\,
		\Res_{023}\,\Res_{23}.\]
	\end{corollary}
	
	It remains to prove
	\begin{lemma}
		$\eta \circ \varphi$ is stable under $\pi_{(2,3)}$.
	\end{lemma}
	\begin{proof} 
		\[\eta \circ \varphi :\Res_{012}\Res_{013}\Zuck_{01}\Ind_{013}\Ind_{012}\rightarrow \sh{-1}{2}[1]\]
		is equal to the composition of two counits and coinductions. Therefore this composition is equal to the following counit map:
		\[\Res_{012,013}\Res_{01}\Rcoind_{01}\coInd_{012,013}\sh{-1}{2}[1] \rightarrow \sh{-1}{2}[1].\]
		This counit is invariant by construction.
	\end{proof}
	
	\begin{corollary}
		\[\pi_{(2,3)}\mathcal{E}\pi_{(2,3)}^{-1}\simeq \mathcal{E}.\]
		Thus,
		\[\pi_{(2,3)}\mathcal{G}\pi_{(2,3)}^{-1}\simeq \mathcal{G}.\]
	\end{corollary}
	
	\begin{theoremC}
		For all pairwise distinct $i,j,k\in\{1,2,3\}$ there is an isomorphism of endofunctors of $D^b(\Ofd(\fp_{123}))$
		\[\mathcal{F}_{ij}\circ\mathcal{F}_{ik}\simeq \mathcal{F}_{ik}\circ \mathcal{F}_{ij}.\]
	\end{theoremC}
	\begin{proof}
		As we have checked,
		\[\pi_{(2,3)}\mathcal{E}\pi_{(2,3)}^{-1}\simeq \mathcal{E}.\]
		Thus,
		\[\pi_{(2,3)}\mathcal{G}\pi_{(2,3)}^{-1}\simeq \mathcal{G}.\]
		In other words, 
		\[\mathcal{F}_{12}\circ\mathcal{F}_{13}\simeq \mathcal{F}_{13}\circ \mathcal{F}_{12}.\]
		The general case follows by conjugation by the permutations $\sigma \in \mathfrak{S}_3$.
	\end{proof}
	
	\section*{Acknowledgements}
	S.A. would like to thank Shamil Shakirov for a multi-year collaboration that resulted in~\cite{ArthamonovShakirov-2019}, which became a prerequisite for the current manuscript. The potential connection between~\cite{ArthamonovShakirov-2019} and the combinatorics of the root system of $\Dalpha$ was brought to the attention of S.A. by Junichi Shiraishi during his visit to BIMSA in December 2024. 
	We are grateful to Sergei Ivanov for useful conversations on properties of triangulated categories.
	The work of S.A. was supported by a grant of the Beijing Natural Science Foundation IS-25025.
	
	\textit{AI statement:} A large variety of AI tools was used during all stages of work on the project. DeepSeek v3.2 and v4, Gemini 3.1 Pro, Kimi K2.6 and K3, and Claude Opus 4.8, Opus 5, Fable 5 and Fable 5.1 were used as coding assistants in C++ and Mathematica for calculations that resulted in the starting character formulas. Claude Opus 5, Claude Fable 5 and Kimi K3 were used while working on the manuscript, but only to improve the clarity and flow of text that had already been written manually. In particular, the TikZ code for Figure~\ref{fig:RootSystem} was produced by K3 based on the blackboard photograph and refined by Claude Opus 5. All proofs and logical content were inserted manually. Before finalizing the manuscript, the use of AI tools was suspended, so that the final version of the text was proofread and edited by the authors alone.
	
	\printbibliography
	
\end{document}